\documentclass{amsart}
\usepackage{graphicx} 
\usepackage{amsmath}
\usepackage{amsthm}
\usepackage{amssymb}
\usepackage{comment}
\usepackage{nameref}
\usepackage{cite}
\numberwithin{equation}{section}

\newcommand{\ul}{U_-}
\newcommand{\ur}{U_+}
\newcommand{\um}{U_m}

\newcommand{\de}{\delta}
\newcommand{\ep}{\epsilon}
\newcommand{\epr}{\delta_{R_n}}
\newcommand{\eps}{\delta_{S_1}}
\newcommand{\epw}{\delta_0}
\newcommand{\epe}{\epsilon_2}
\newcommand{\ept}{\epsilon_0}
\newcommand{\epl}{\epsilon_1}

\newcommand{\utl}{\Tilde{U}}
\newcommand{\p}{\mathbf{E}}

\newcommand{\x}{\mathbf{X}}
\newcommand{\rt}{R_n}
\newcommand{\sis}{\sigma}
\newcommand{\sio}{{\sigma}_1}
\newcommand{\sit}{\sigma_n}

\newcommand{\cf}{c_f^{(1)}}
\newcommand{\cfs}{c_f^{(n)}}
\newcommand{\cfi}{c_f^{(i)}}

\newcommand{\so}{S_1^{\mathbf{X}_1}}
\newcommand{\st}{S_n^{\mathbf{X}_n}}
\newcommand{\si}{S_i^{\mathbf{X}_i}}
\newcommand{\wn}{W_n^{\x_n}}

\newcommand{\dkoo}{\;\partial_{x_1}{k_{S_1}^{\x_1}}}

\newcommand{\dktt}{\;\partial_{x_1}{k_{S_n}^{\x_n}}}

\newcommand{\dkk}{\;\partial_{x_1}{k_{S_i}^{\x_i}}}
\newcommand{\epww}{\delta_{W_n}}
\newcommand{\epws}{\delta_{S_n}}
\newcommand{\epss}{\delta_{S_i}}

\newcommand{\dkww}{\;\partial_{x_1}{k_{W_n}^{\x_n}}}
\newcommand{\dkrr}{\;\partial_{x_1}{k_{R_n}}}
\newcommand{\asxo}{a_{S_1} ^ {\mathbf{X}_1}}

\newcommand{\aw}{a_{W_n}^{\x_n}}
\newcommand{\mn}{\Tilde{\mu}}
\newcommand{\ro}{\mathbf{r_1}}
\newcommand{\rn}{\mathbf{r_n}}
\newcommand{\lo}{\mathbf{l_1}}
\newcommand{\lnn}{\mathbf{l_n}}
\newcommand{\li}{\mathbf{l_i}}
\newcommand{\voi}{\mathbf{v_i^{(1)}}}
\newcommand{\vti}{\mathbf{v_i^{(n)}}}

\newcommand{\stdy}{\mathbf{S}_i}
\newcommand{\sdy}{{\sigma}_{i}}
\newcommand{\uls}{U_L}
\newcommand{\urs}{U_R}

\newtheorem{theo}{Theorem}[section]
\newtheorem{lemm}[theo]{Lemma}
\newtheorem{prop}[theo]{Proposition}
\newtheorem{assum}[theo]{Assumption}

\title[Time-asymptotic stability of composite planar waves in multi-D]{Time-asymptotic stability of composite weak planar waves for a general $n\times n$ multi-D viscous system}
\date{\today}
\author{Jiayun Meng}
\address{Department of Mathematics, \newline The University of Texas at Austin, Austin, TX 78712, USA}
\thanks{We would like to thank our advisor Alexis Vasseur for suggesting this problem and for valuable discussions and encouragements during the completion of this project.}
\thanks{\textbf{Acknowledgment.}This work was partially funded by NSF-DMS 2306852, NSF-RTG 1840314, and NSF-EPSRC 2219434.}
\email{jiayun@utexas.edu}

\begin{document}

\begin{abstract}
    We prove the time-asymptotic stability of the superposition of a weak planar viscous 1-shock and either a weak planar n-rarefaction or a weak planar viscous n-shock for a general $n\times n$ multi-D viscous system. 
    In 2023, Kang-Vasseur-Wang \cite{MVWns} showed the stability of the superposition of a viscous shock and a rarefaction for 1-D compressible barotropic Navier-Stokes equations and solved a long-standing open problem officially introduced by Matsumura-Nishihara \cite{MNht} in 1992. Our work is an extension of \cite{MVWns}, where a general $n\times n$ multi-D viscous system is studied.
    Same as in \cite{MVWns}, we apply the $a$-contraction method with shifts, an energy based method invented by Kang and Vasseur in \cite{KVa}, for both viscous shock and rarefaction at the level of the solution. In such a way, we can work with general perturbations and compositions of waves. Finally, a technique to classify and control higher-order terms is developed to work in multi-D. 
\end{abstract}
\maketitle
\tableofcontents
\section{Introduction}
We consider a general $n\times n$ multi-D viscous system
\begin{equation} \label{eq for sol}
    \partial_t U + \partial_{x_1} f(U) + \sum_{j=2}^{d} \partial_{x_j} g_j(U) = \sum_{j=1}^{d} \partial_{x_j} \bigl(B_j(U)\partial_{x_j} \eta'(U)\bigl)\, ,  
\end{equation}
where $U: \mathbb{R}_{\geq 0}\times \mathbb{R}\times \mathbb{T}^{d-1}\rightarrow \mathcal{V} \subseteq \mathbb{R}^n$ with an open convex phase space $\mathcal{V}$. We assume the periodic boundary condition in the transverse direction $\mathbb{T}^{d-1}$. The notation for the physical space is $x = (x_1,y)$ with $y$ denoting the transverse direction $\mathbb{T}^{d-1}$.
The flux functions $f, g_j: \mathcal{V}\rightarrow \mathbb{R}^n$ are assumed to be smooth. In addition, the flux function $f$ is assumed to be strictly hyperbolic and genuinely nonlinear. 
The viscosity coefficient matrix $B_j: \mathcal{V} \rightarrow \mathbb{R}^{n\times n}$ is assumed to be smooth. For any $V\in\mathcal{V}$, $B_j$ is assumed to be an $n\times n$ positive definite matrix. We assume that the entropy $\eta: \mathcal{V}\rightarrow \mathbb{R}$ is smooth and strictly convex, and for any $1\leq j\leq d$, there exists an entropy flux $q_j: \mathcal{V}\rightarrow \mathbb{R}$ of $\eta$ such that 
\begin{equation}\label{def of entropy flux}
    \partial_i q_j(U) = 
    \begin{cases}
        \sum_{l=1}^n \partial_l \eta(U) \partial_i (f)_l (U) \, , \textup{ if } j=1 \, , \\
        \sum_{l=1}^n \partial_l \eta(U) \partial_i (g_j)_l (U) \, , \textup{ if } 2\leq j\leq d \, ,
    \end{cases}
\end{equation}
for any $1\leq i\leq n$.

Let us endow the system with an initial value 
\begin{equation*}
    U(0, x) = U_0(x)\, , \quad x\in\mathbb{R}\times \mathbb{T}^{d-1}  \, ,
\end{equation*}
with fixed end states $U_\pm \in \mathbb{R}^n$, i.e., 
\begin{equation} \label{initial value large time}
    U_0(x) \rightarrow U_\pm\, , \quad \textup{as } x_1\rightarrow \pm \infty \, .
\end{equation}
This general framework includes the 3-D barotropic Brenner-Navier-Stokes equations:
\begin{equation}\label{BNS eqs}
    \begin{cases}
        \partial_t \rho + div(\rho v) = 0 \, , \\
        \partial_t \rho u + div(\rho u \otimes v) + \nabla \rho^\gamma = \nu\Delta u \, ,
    \end{cases}
\end{equation}
where $\nu>0$, $\gamma>1$, $\rho$ is the density, $u$ is the velocity, and $v$ is the corrected velocity defined by
\begin{align*}
    v = u - \frac{1}{\rho} \nabla Q'(\rho)\, , \quad \textup{where  } Q(\rho) = \frac{\rho^\gamma}{\gamma-1} \, .
\end{align*}
Such a system was introduced by Brenner as a correction of the Navier-Stokes equations. In \cite{BNS}, Brenner mentioned some extreme situations where the Navier-Stokes equations did not describe the compressible fluid well, and based on those results, Brenner proposed that the specific momentum density of the fluid equals the corrected velocity $v$ (the volume velocity) instead of the velocity $u$ (the mass velocity).

In this paper, we study the long-time asymptotic behavior of solutions of (\ref{eq for sol}) with initial value satisfying (\ref{initial value large time}). This long-time asymptotic behavior is closely related to the following 1-D Riemann problem:
\begin{equation}\label{eq for sol inviscid}
    \partial_t U_1 + \partial_{x_1} f(U_1) = 0\, , 
\end{equation}
with an associated initial value
\begin{equation}\label{inviscid initial value}
    U_1(0, x_1) = 
    \begin{cases}
         U_-  \, , \quad x_1<0 \, ,\\
         U_+  \, , \quad x_1>0 \, .
    \end{cases}
\end{equation}

\noindent
\textbf{Elementary solutions of the 1-D Riemann problem.} 
Since $f$ is genuinely nonlinear, (\ref{eq for sol inviscid}) has two elementary solutions: shock wave and rarefaction wave. 
We denote the eigenvalues of $f'$ as $\lambda_1<...<\lambda_n$ and the corresponding right eigenvectors as $r_1,...,r_n$. 
Let $\ul \in \mathbb{R}^n$. 

For any $1\leq i\leq n$, there exists an integral curve $ \mathcal{R}_i(\ul)$ such that for any $\ur \in \mathcal{R}_i(\ul)$ that is close enough to $\ul$, there exists a solution $\mathbf{R}_i$ to (\ref{eq for sol inviscid}) with initial value (\ref{inviscid initial value}) that is defined by
\begin{equation*}
    \lambda_i\bigl(\mathbf{R}_i(t,x_1)\bigl) = 
    \begin{cases}
        \lambda_i(\ul) \, , \quad x_1<\lambda_i(\ul) t \, , \\
        \frac{x_1}{t} \, , \quad \lambda_i(\ul) t \leq x_1 \leq \lambda_i(\ur) t \, , \\
        \lambda_i(\ur) \, , \quad x_1 > \lambda_i(\ur) t \, , 
    \end{cases}
\end{equation*}
with
\begin{equation*}
    z_i\bigl(\mathbf{R}_i(t,x)\bigl) = z_i(\ul) = z_i(\ur) \, ,
\end{equation*}
for any $i$-Riemann invariant $z_i$. The existence of $(n-1)$ $n$-Riemann invariants can be found in \cite{Sinvariant}. 
We call $\mathbf{R}_i$ the i-rarefaction wave. 

We define the shock set 
\begin{equation*}
    \mathcal{S}(\ul) = \{U | -\sis (U-\ul) + f(U)-f(\ul) = 0 \text{ for a constant } \sis\}\, .
\end{equation*}
In a neighborhood of $\ul$, $\mathcal{S}$ consists of n smooth curves $\mathcal{S}_1(\ul), ..., \mathcal{S}_n(\ul)$ that intersects at $\ul$. We call $\mathcal{S}_i(\ul)$ the i-shock curve.
For any $\ur \in \mathcal{S}_i(\ul)$ that is close enough to $\ul$, there exists a solution $\Tilde{\mathbf{S}}_i$ to (\ref{eq for sol inviscid}) with initial value (\ref{inviscid initial value}) that is defined by
\begin{equation*}
    \Tilde{\mathbf{S}}_i(t,x_1) =
    \begin{cases}
        \ul \, , \quad x_1<\sis t \, , \\
        \ur \, , \quad x_1>\sis t \, . \\
    \end{cases}
\end{equation*}
We call $\Tilde{\mathbf{S}}_i(\ul)$ the i-shock wave and $\sis$ the speed of the i-shock wave. 
In particular, $\sis$ is close to $\lambda_i(\ul)$. We denote $\sis$ as $\sigma_i$.

\noindent
\textbf{Stability of elementary solutions for 1-D viscous systems}. If the Riemann solution to an inviscid system is a rarefaction wave (respectively a shock wave), then the asymptotic state of the solution to the corresponding viscous system with a perturbed initial value is the rarefaction wave (respectively the viscous shock wave). 

While the rarefaction wave spreads much faster than the diffusion process, the jump of the shock is smoothed by the viscosity and becomes a thin transition layer. The viscous i-shock wave $\mathbf{S}_i(x-\sis_i t)$ refers to the smoothened i-shock wave, and $\mathbf{S}_i$ satisfies 
\begin{equation}\label{eq for sol shock}
    -\sis_i \partial_{x_1} \mathbf{S}_i + \partial_{x_1} f(\mathbf{S}_i) = \partial_{x_1} \bigl(B(\mathbf{S}_i)\partial_{x_1} \eta'(\mathbf{S}_i)\bigl)\, .
\end{equation}
Ilin-Oleinik \cite{IOscalar} proved the stability of elementary solutions for the scalar case in 1960, but the maximum principle they used did not work for systems (see \cite{Mht}). Later, the energy method has become the main tool. When studying the stability of rarefaction waves, the energy method can be applied directly at the level of the solutions. We call such a method the direct energy method. The stability of rarefaction waves was first proved by Matsumura-Nishihara \cite{MNho,MNht} for the 1-D compressible Navier-Stokes equations. Later, Liu-Xin \cite{LXstbltyR} and Nishihara-Yang-Zhao \cite{NYZstablyR} pushed the stability result of rarefaction waves to the Navier-Stokes-Fourier system.

However, the direct energy method fails when working with viscous shock waves. The first results were based on the energy method applied at the level of the antiderivative. Matsumura-Nishihara \cite{MNvs} in 1985, and Goodman \cite{Gvs} in 1986 independently proved the stability of viscous shock waves. Matsumura and Nishihara proved the stability for the 1-D compressible barotropic Navier-Stokes system, while Goodman proved the stability for a general 1-D system with a positive definite viscosity. As both papers worked with the antiderivative, they need to assume the zero mass condition, i.e., that the mass of the initial perturbation is zero.

Two very fruitful approaches removed the stringent zero mass condition. The Green's function method was started by Liu \cite{Liu1985} in 1985. It involves a constant shift on the viscous shock and the introduction of a diffusion wave and a coupled diffusion wave in the transverse characteristic fields.
Szepessy-Xin \cite{SXgreen} showed the stability of viscous shock waves for a 1-D general system with a nondegenerate artificial viscosity, and Liu-Zeng \cite{LZgreen} applied the Green's function method to 1-D systems with degenerate viscosity.
Another approach is the Evans function method. In 2004, Mascia-Zumbrun \cite{MZspectral} showed the spectral stability of viscous shock waves for the 1-D compressible Navier-Stokes system. 
The study of spectral stability is very advanced now. In 2017, Humpherys-Lyng-Zumbrun \cite{HLZspectral} proved the spectral stability of large-amplitude planar viscous shock waves for the compressible Navier-Stokes equations in multi-D by the numerical Evans function method. Although this technique gives stability only for single elementary waves, the stability result works in the whole space $\mathbb{R}^d$.

Both Green's function method and Evans function method can deal with more general perturbations and provide pointwise estimates but fail to give global-in-time stability results for composite waves. Assuming the strengths of the two viscous shock waves are suitably small with the same order,
Huang-Matsumura \cite{HMtwovs} showed the stability of two viscous shock waves for the 1-D Navier-Stokes-Fourier equations under a more relaxed condition than the zero mass condition.

\noindent
\textbf{Matsumura Conjecture}. Even in 1-D, the stability of the composition of a viscous shock wave and a rarefaction wave is very difficult to study. Matsumura-Nishihara \cite{MNho} mentioned the problem in 1986 and officially introduced it as an open problem in 1992 in \cite{MNht}. In 2018, Matsumura \cite{Mht} classified the problem as a very hard open problem. There are two difficulties. First, the direct energy method used to study the stability of rarefaction waves does not match very well with the methods developed for viscous shock waves. At the same time, the rarefaction wave is not an exact solution to the viscous system and any spatial shift of the rarefaction wave has the same asymptotic state. Hence, it is hard to analyze the interaction between the rarefaction wave and the viscous shock wave.

The breakthrough happened in 2023. Using the $a$-contraction method with shifts, 
Kang-Vasseur-Wang proved the stability of the composition of a viscous shock wave and a rarefaction wave for the 1-D compressible barotropic Navier-Stokes equations in \cite{MVWns} and the stability of the generic Riemann solutions for the 1-D compressible Navier-Stokes-Fourier equations in \cite{KVWnsfgeneric}.
The $a$-contraction method with shifts is an energy method that can be applied at the level of the solution for contact waves, rarefaction waves, and viscous shock waves, so it unifies the methods for elementary waves. This paper is an extension of their work \cite{MVWns}, in which the stability of a planar viscous 1-shock wave and either a planar n-rarefaction wave or a planar viscous n-shock wave is proved for a general $n\times n$ multi-D viscous system. Note that we consider here only extremal waves.

\noindent
\textbf{Multi-D results}. 
The stability of planar rarefaction waves for the 3-D compressible Navier-Stokes-Fourier system was shown by Li-Wang-Wang \cite{LWWstabilityR} in 2018. 
In 2023, the stability of planar viscous shock waves for the 3-D compressible Navier-Stokes equations was proved by Wang-Wang \cite{WWstabilitys}. Later in 2024, Kang-Lee \cite{KLstabibilitys} generalized Wang-Wang's stability result to the compositions of planar viscous shocks for the same system. 
Note that all these multi-D results are based on a-contraction method with shifts and work only with periodic transversal variables and weak elementary waves. 

\noindent
\textbf{Result of the paper}.
Let $m$ be the smallest integer that is strictly bigger than $\frac{d}{2}$, i.e., 
\begin{equation}\label{value of m}
    m=
    \begin{cases}
        \frac{d+1}{2} \, , \textup{ if } d \textup{ is odd,} \\
        \frac{d+2}{2} \, , \textup{ if } d \textup{ is even.}
    \end{cases}
\end{equation}
The main result of the paper is the following theorem.
\begin{theo}\label{thm}
For any $\ul \in \mathbb{R}^n$, there exist constants $\epw, \ept>0$ such that the following is true.

\noindent
Let $\um \in \mathcal{S}_1(\ul)$ and $\ur \in \mathcal{S}_n(\um)$ or $\mathcal{R}_n(\um)$ be such that
\begin{equation*}
    |\um-\ul| + |\ur-\um| < \epw \, .
\end{equation*}
Let $\mathbf{S}_1$ be the viscous 1-shock wave solution to (\ref{eq for sol shock}) with end states $\ul$ and $\um$, and $\mathbf{W}_n$ be the viscous n-shock wave solution $\mathbf{S}_n$ to (\ref{eq for sol shock}) or the n-rarefaction wave solution $\mathbf{R}_n$ to (\ref{eq for sol inviscid}) with end states $\um$ and $\ur$. Let $\mathcal{I}_{\mathbf{W}_n}$ be $\{1,n\}$ if $\mathbf{W}_n=\mathbf{S}_n$ and be $\{1\}$ if $\mathbf{W}_n=\mathbf{R}_n$.

\noindent
Let $U_0$ be an initial value such that 
\begin{equation}\label{thm initial data}
    \sum_\pm \|U_0-U_\pm\|_{L^2(\mathbb{R}_\pm\times \mathbb{T}^{d-1})} + \|U_0\|_{\dot{H}^m(\mathbb{R}\times\mathbb{T}^{d-1})}< \ept \, ,
\end{equation}
where $\mathbb{R}_+=-\mathbb{R}_-=(0,\infty)$.

\noindent
Then the viscous system (\ref{eq for sol}) has a unique global-in-time solution $U$. Moreover, there exist absolutely continuous shifts $\x_i(t)$ for $i\in\mathcal{I}_{\mathbf{W}_n}$ such that
\begin{align}
    &U(t,x) - \mathbf{\utl}(t,x) \in C\bigl(0,\infty;H^m(\mathbb{R}\times\mathbb{T}^{d-1})\bigl) \, , \label{thm continuity} \\
    &\underset{x\in \mathbb{R}\times\mathbb{T}^{d-1}}{\textup{sup}}\, |U(t,x)- \mathbf{\utl}(t,x)| \rightarrow 0 \; \text{ as } \; t\rightarrow 0 \, , \label{thm pointwise convergence}
\end{align}
where
\begin{equation*}
    \mathbf{\utl}(t,x) = 
    \begin{cases}
        \mathbf{S}_1\bigl(x_1-\sio t + \x_1(t)\bigl) + \mathbf{S}_n\bigl(x_1-\sit t + \x_n(t)\bigl) -\um \, , \textup{ if } \mathbf{W}_n=\mathbf{S}_n \, ,\\
        \mathbf{S}_1\bigl(x_1-\sio t + \x_1(t)\bigl) + \mathbf{R}_n\bigl(\frac{x_1}{t}\bigl) -\um \, , \textup{ if } \mathbf{W}_n=\mathbf{R}_n \, .
    \end{cases}
\end{equation*}
In addition,
\begin{align}
    U -\mathbf{\utl} \in L^2\bigl(0,\infty;\dot{H}^{m+1}(\mathbb{R}\times\mathbb{T}^{d-1})\bigl) \, , \label{thm second deri}
\end{align}
and for any $i \in \mathcal{I}_{\mathbf{W}_n}$,
\begin{equation} \label{thm X shift}
    \underset{t\rightarrow \infty}{\textup{lim}} \dot{\mathbf{X}}_i(t)= 0 \, .
\end{equation}
\end{theo}
\noindent
\textit{Remark.} Theorem \ref{thm} shows that if $\ul$ and $\ur$ are connected by a composition of a planar viscous 1-shock wave and either a planar viscous n-shock wave or a planar n-rarefaction wave, then the asymptotic state of the solution to the viscous system (\ref{eq for sol}) with a perturbed initial value is the composition with viscous shocks shifted. 
\begin{prop}\label{prop BNS}
The 3-D barotropic Brenner-Navier-Stokes equations (\ref{BNS eqs}) can be transformed into the form of the viscous system (\ref{eq for sol}).
\end{prop}
\noindent
\textit{Remark.} The detailed transformation is in the appendix. Proposition \ref{prop BNS} implies that Theorem \ref{thm} can be applied to the 3-D barotropic Brenner-Navier-Stokes equations. 

\noindent
\textbf{Structure of the paper}. 
Section \ref{sec preliminaries} discusses the properties of the viscous shock wave and the approximate rarefaction wave and introduces the weight functions, shift functions, and the superposition wave. 

In section \ref{proof of theorem}, we show how Theorem \ref{thm} is proved by local-in-time estimates and a priori estimates. Proposition \ref{existence prop} provides local-in-time estimates which could be shown in the same way as previous work, while a priori estimates are given in Proposition \ref{prop energy estimate} which is to be shown in sections \ref{section energy estimates} and \ref{H^m contraction}. 

We get the $L^2$ estimates by the $a$-contraction method with shifts in section \ref{section energy estimates} and go from the $L^2$ estimates to the $H^m$ estimates (a priori estimates) by induction in section \ref{H^m contraction}.

\noindent
\textbf{The $a$-contraction method with shifts}. 
The method of $a$-contraction relies on the ad-hoc construction of the shifts $\x_i$ for $i\in \mathcal{I}_{\mathbf{W}_n}$ solving special ODEs. For the scalar case, the method can be applied directly on the $L^2$ norm (see \cite{KVa}). However, it was shown that the result is not true for systems in \cite{SVsys}. To work on systems, we need to introduce weight functions $a_{\mathbf{S}_1}, a_{\mathbf{W}_n}$. In this paper, we follow the method of \cite{MVWns} written for the special case of the 1-D compressible barotropic Navier-Stokes equations. Our extension allows to clarify and explain why the method works at a deeper level. 

To get the $L^2$ estimates, we study the evolution of a pseudo-distance given by the physical structure of the problem
\begin{align*}
    \frac{d}{dt}\int (a_{\mathbf{S}_1} + a_{\mathbf{W}_n}) \, \eta(U|\mathbf{\utl}) \, dx \, ,
\end{align*}
where the relative entropy $\eta(\cdot|\cdot)$ is defined by
\begin{equation}\label{eq for REF}
    \eta(U|V) = \eta(U) - \eta(V) - \eta'(V)(U-V) \, .
\end{equation}
We define two bases at the beginning of subsection \ref{subsec relative entropy}. The basis (\ref{decomposition of S_1}) (respectively (\ref{decomposition of W})) is designed for the wave $\mathbf{S}_1$ (respectively $\mathbf{W}_n$). It contains the special direction $r_1(\ul)$ (respectively $r_n(\ul)$) of the wave $\mathbf{S}_1$ (respectively $\mathbf{W}_n$) and is orthogonal with respect to the viscous matrix $B_1$ in order to be consistent with the viscosity.

In Lemma \ref{lemma rem}, we apply the relative entropy method introduced by Dafermos \cite{Darem} and DiPerna \cite{Direm} and project the perturbation $U-\mathbf{\utl}$ onto the bases. We get
\begin{align*}
    \frac{d}{dt}\int (a_{\mathbf{S}_1} + a_{\mathbf{W}_n}) \, \eta(U|\utl) \, dx \leq \mathcal{Z}(U) - \mathcal{D}(U) + \mathcal{H}(U) + \mathcal{E}(U) \, .
\end{align*}
The shift term $\mathcal{Z}$ represents the new terms induced by the shifts $\x_i$ for $i\in \mathcal{I}_{\mathbf{W}_n}$. The viscosity operation (the right-hand side of (\ref{eq for sol})) gives the viscous term $\mathcal{D}$. The hyperbolic term $\mathcal{H}$ comes from the flux functions $f,g_j$. The interaction between waves $\mathbf{S}_1$ and $\mathbf{W}_n$ creates the interaction term $\mathcal{E}$.

Lemma \ref{lemma rem} discusses the hyperbolic ``scalarization".
The weight function $a_{\mathbf{S}_1}$ (respectively $a_{\mathbf{W}_n}$) activates the spectral gap, creates new negative hyperbolic terms, and initiates cancellation for hyperbolic terms corresponding to the perturbation $U-\mathbf{\utl}$ in all directions except the special direction $r_1(\ul)$ (respectively $r_n(\ul)$) of the wave $\mathbf{S}_1$ (respectively $\mathbf{W}_n$). We get
\begin{align*}
    \mathcal{H}(U) \leq -C_{a} \mathcal{H}_C(U) + C \mathcal{H}_{\mathbf{S}_1}(U) + C \mathcal{H}_{\mathbf{W}_n}(U) \, ,
\end{align*}
where $C_a>0$ is a constant that depends on the strengths of $a_{\mathbf{S}_1}, a_{\mathbf{W}_n}$, and $C>0$ represents the constants that depend only on $B_j,f,\eta,\ul$. $\mathcal{H}_{\mathbf{S}_1}$ (respectively $\mathcal{H}_{\mathbf{W}_n}$) is the hyperbolic term corresponding to the projection of the perturbation in the special direction $r_1(\ul)$ (respectively $r_n(\ul)$) of the wave $\mathbf{S}_1$ (respectively $\mathbf{W}_n$), and $\mathcal{H}_{C}$ denotes the sum of the absolute value of the hyperbolic terms corresponding to the other orthogonal directions of the perturbation. We see the hyperbolic ``scalarization" scalarizes the problem to the special direction $r_1(\ul)$ (respectively $r_n(\ul)$) of the wave $\mathbf{S}_1$ (respectively $\mathbf{W}_n$).

The hyperbolic remainder due to the perturbation in the special direction $r_n(\ul)$ of the rarefaction wave $\mathbf{R}_n$ is negative, so the rarefaction wave $\mathbf{R}_n$ is contractive at the hyperbolic level. However, the hyperbolic remainder due to the perturbation in the special direction of a viscous shock wave is positive and needs to be depleted using the viscous term $\mathcal{D}$. This is done by introducing a Poincar\'{e} type inequality in Lemma \ref{lemma parabolic}. 
For $i\in \mathcal{I}_{\mathbf{W}_n}$, we show 
\begin{align*}
    \mathcal{H}_{\mathbf{S}_i}(U) \lesssim \mathcal{D}(U) - \mathcal{Z}(U) + \mathcal{H}_C(U)\, .
\end{align*}
The viscous term $\mathcal{D}$ controls the $L^2$ norm of the derivative, while the strengths of the shifts $\x_i$ for $i\in \mathcal{I}_{\mathbf{W}_n}$ are chosen to be big enough that the shift term $\mathcal{Z}$ handles the average with the help of $\mathcal{H}_C$. 

In Lemma \ref{lemma L^2 estimate}, we choose weight functions $a_{\mathbf{S}_1}, a_{\mathbf{W}_n}$ and shift functions $\x_i$ for $i\in \mathcal{I}_{\mathbf{W}_n}$ that work for both the hyperbolic ``scalarization" and the Poincar\'{e} type inequality. We get
\begin{align*}
    \mathcal{Z}(U) - \mathcal{D}(U) + \mathcal{H}(U) \leq 0 \, .
\end{align*}
By Lemma \ref{property of rarefaction} and Lemma \ref{interaction bound}, we can be bound the interaction term $\mathcal{E}(U)$ by a small and time integrable function depending on the strengths of the waves $\mathbf{S}_1, \mathbf{W}_n$. In all, we obtain the $L^2$ estimates at the end of section \ref{section energy estimates}. 

\noindent
\textit{Remark}. In Lemma \ref{lemma parabolic}, we reduce the Poincar\'{e} type inequality to the following lemma proved in \cite{KVpoincare}.
\begin{lemm}\label{Poincare inequality lemma}
    For any $f:[0,1]\rightarrow\mathbb{R}$ satisfying $\int_0^1 y (1-y)|f'|^2 \, dy < \infty$,
    \begin{equation}
        \int_0^1 |f-\int_0^1 f \, dy|^2 \, dy \leq \frac{1}{2} \int_0^1 y (1-y)|f'|^2 \, dy \, .
    \end{equation}
\end{lemm}


\noindent
\textbf{Notations and a remark.} Before we go to the proofs, let us fix some notations. Let the eigenvalues of $f'(\ul)$ be $\lambda_1<...<\lambda_n$. For any $1\leq i\leq n$, let $\mathbf{r_i}$ and $\li$ be the right and left eigenvectors corresponding to $\lambda_i$ such that $\mathbf{r_i}$ is tangent to $\mathcal{S}_i(\ul)$, $\li=\eta''(\ul)\mathbf{r_i}$, and $\mathbf{r_i}\cdot \li=1$.
We define
\begin{equation} \label{r curve opposite s curve}
    \begin{aligned}
        \cfi :=( f''(\ul) : \mathbf{r_i} \otimes \mathbf{r_i}) \cdot \li &= \lambda_i '(\ul) \cdot \mathbf{r_i} < 0 \, .
    \end{aligned}
\end{equation}
Let $C$ denote the positive constants that depend only on $B_j, f, \eta, \ul$. For $\alpha = (a_1,...,a_d)\in \mathbb{N}^d$, we define
\begin{align}\label{notation of alpha}
    (\alpha)_j = a_j \, \textup{ for any }\,  1\leq j\leq d \; \textup{ and }\;  \partial_x^{\alpha} = \partial_{x_1}^{a_1}...\partial_{x_d}^{a_d} \, .
\end{align}
As $\eta'' f'$ is symmetry, we know
\begin{equation} \label{trick eta'' f'}
    \eta''(\ul) \mathbf{r_i} \cdot \mathbf{r_j} = 0 \text{ if } i\neq j \, .
\end{equation}

\section{Preliminaries}\label{sec preliminaries}
\subsection{Viscous shock wave}
Let $i\in \{1,n\}$. We examine the viscous i-shock wave $\stdy$ satisfying
\begin{equation} \label{eq for viscous shock}
    \begin{cases}
        -\sdy \partial_{x_1} \stdy + \partial_{x_1} f(\stdy) = \partial_{x_1}\bigl(B_1(\stdy)\partial_{x_1} \eta'(\stdy)\bigl) \, , \\
        \stdy(-\infty) = \uls \, ,\; \stdy(+\infty) = \urs \, .
    \end{cases}
\end{equation}
Recall the Rankine-Hugoniot condition and the Lax inequality 
\begin{equation*}
    \lambda_1(\urs)< \sdy< \lambda_1(\uls) \, .
\end{equation*}
Let the wave strength of $\stdy$ be
\begin{align*}
    \epss = (\urs - \uls) \cdot \li \, .
\end{align*}
The proof of the existence of the viscous i-shock wave can be found in \cite{MajdaSexist}.
The following results are proved in \cite{Preparation}. 
\begin{lemm}\label{pointwise bound of shock}
    For any $\uls \in \mathbb{R}^n$, there exist $\ep, C>0$ and $C_j>0$ for any $j\geq 2$ such that the following is true.

    \noindent
    For any $\urs \in \mathcal{S}_i(\uls)$ such that $|\uls-\urs| < \ep$, there exists a unique solution $\stdy$ to (\ref{eq for viscous shock}) such that
    \begin{align*}
        |\stdy({x_1})-\uls| &\leq C \epss e^{-C\epss|{x_1}|}, \quad \forall {x_1}< 0 \, , \\
        |\stdy({x_1})-\urs| &\leq C \epss e^{-C\epss|{x_1}|}, \quad \forall {x_1}> 0 \, , \\
        |\partial_{x_1} \stdy| &\leq C \epss^2 e^{-C\epss|{x_1}|}, \quad \forall {x_1} \in \mathbb{R}\, , \\
        |\partial_{x_1}^j \stdy| &\leq C_j \epss |\partial_{x_1} \stdy|, \quad \forall {x_1}\in \mathbb{R}\, , \; \forall j \geq 2 \, .
    \end{align*}
\end{lemm}
We define 
\begin{align*}
    S_i(t,x) = \stdy(x_1 - \sigma_i t) \, ,
\end{align*}
and the projection of $S_i$ onto $\li$
\begin{equation}\label{decomposition of S}
    \begin{aligned}
        k_{S_i}(t,x) = \frac{\bigl(S_i(t,x) - \uls\bigl) \cdot \li}{\epss} \,.
    \end{aligned}
\end{equation}
From now on, we call $S_i$ the planar viscous i-shock wave. 
We know $S_i$ satisfies
\begin{align}\label{eq for shock}
    \partial_t S_i + \partial_{x_1} f(S_i) &= \partial_{x_1}\bigl(B_1(S_i)\partial_{x_1} \eta'(S_i)\bigl) \, .
\end{align}
In \cite{Preparation}, we show
\begin{align}\label{derivative of k}
    \big|\partial_{x_1}k_{S_i} +\frac{\cf}{2B(\ul)\li \cdot \li} \epss k_{S_i} (1-k_{S_i}) \big| \leq C \epss ^ {2} k_{S_i} (1-k_{S_i}) \, ,
\end{align}
and
\begin{equation}\label{k_S as main part}
    \begin{aligned}
        |\partial_{x_1} S_i-\epss \partial_{x_1} k_{S_i} \mathbf{r_i}| \leq C \epss^{2} \partial_{x_1} k_{S_i} \, .
    \end{aligned}
\end{equation}
In particular, $k_{S_i}$ is strictly increasing.

\subsection{Construction of approximate rarefaction wave}
As in \cite{MVWns}, we will consider a smooth approximation of the planar n-rarefaction wave with the help of the smooth solution to the Burgers' equation
\begin{equation}\label{burgers' equation}
    \begin{cases}
        w_t + ww_{x_1}=0 \, , \\
        w(0,x_1)=w_0(x_1)=\frac{w_++w_m}{2}+\frac{w_+-w_m}{2}\text{tanh} x \, .
    \end{cases}
\end{equation}
The smooth approximate planar n-rarefaction wave $\rt$ is defined by
\begin{equation} \label{def of rarefaction}
    \begin{aligned}
        &\lambda_n(\um) = w_m \, , \; \lambda_n(\ur)=w_+ \, , \\
        &\lambda_n\bigl(\rt(t,x)\bigl) = w(1+t,x_1) \, , \\
        &z_n\bigl(\rt(t,x)\bigl)=z_n(\um)=z_n(\ur)\, , 
    \end{aligned}
\end{equation}
where $w(t,x_1)$ is the smooth solution to the Burgers' equation (\ref{burgers' equation}) and $z_n$ is any $n$-Riemann invariant to (\ref{eq for sol inviscid}).

It is easy to check that $\rt$ is the solution to the inviscid system, i.e., 
\begin{equation} \label{eq for rarefaction}
    \partial_t \rt + \partial_{x_1} f(\rt) = 0 \, .
\end{equation}
We define the wave strength of the rarefaction 
\begin{equation*}
    \epr = -(\ur-\um) \cdot \lnn  \, ,
\end{equation*}
and
\begin{align}\label{decomposition of R}
    k_{R_n} = \frac{-(R_n-\um) \cdot \lnn}{\epr} \, .
\end{align}
The following properties of the approximate planar n-rarefaction wave $R_n$ follow from the properties of the smooth solution to the Burgers' equations proved in \cite{MNho}. 
\begin{lemm} \label{property of rarefaction}
    The smooth approximate n-rarefaction wave $\rt$ defined in (\ref{def of rarefaction}) satisfies the following properties. 
    \begin{itemize}
        \item[1)] $\partial_{x_1} \rt \cdot \lnn > 0$ and $|\partial_{x_1} \rt \cdot \li| \leq C \epr \partial_{x_1} \rt \cdot \lnn$ for any $1\leq i\leq n-1$.
        \item[2)] For any $t\geq 0$ and $j\in \mathbb{N}^*$,
            \begin{align*}
                \|\partial_{x_1}^j \rt\|_{L^p(\mathbb{R}\times\mathbb{T}^{d-1})} &\leq C_{p,j} \, \textup{min}\bigl\{\epr, \epr^{1/p}(1+t)^{-1+1/p}\bigl\} \, , \; \forall p \in [1,\infty] \, ,\\
                \|\partial_{x_1}^{2j} \rt\|_{L^p(\mathbb{R}\times\mathbb{T}^{d-1})} &\leq C_{p,j} \, \textup{min}\bigl\{\epr, (1+t)^{-1}\bigl\} \, , \;\forall p \in [1,\infty) \, ,\\
                |\partial_{x_1}^j \rt| &\leq C_{j} |\partial_{x_1} \rt| \,  ,
            \end{align*}
            where $C_j>0$ depends on $j$ and $C_{p,j}>0$ depends on $p, j$.
        \item[3)] For any $t\geq 0$,
            \begin{align*}
                |\rt(t,x)-\um|&\leq C \epr e^{-2|x_1-\lambda_n(\um)t|} \, , \quad \forall x_1 \leq \lambda_n(\um)t \, , \\
                |\partial_{x_1} R_n(t,x)| &\leq C \epr e^{-2|x_1-\lambda_n(\um)t|} \, , \quad \forall x_1 \leq \lambda_n(\um)t  \, , \\
                |\rt(t,x)-\ur|&\leq C \epr e^{-2|x_1-\lambda_n(\ur)t|} \, , \quad \forall x_1 \geq \lambda_n(\ur)t \, , \\
                |\partial_{x_1} \rt(t,x)| &\leq C \epr e^{-2|x_1-\lambda_n(\ur)t|} \, , \quad \forall x_1 \geq \lambda_n(\ur)t  \, .
            \end{align*}
        \item[4)] $\underset{t\rightarrow \infty}{\textup{lim}}\,\underset{x\in\mathbb{R}\times\mathbb{T}^{d-1}}{\textup{sup}} \, |\rt(t,x)-\mathbf{R}_n(\frac{x_1}{t})| = 0 \, .$ 
    \end{itemize}
\end{lemm}
In particular, $k_{R_n}$ is strictly increasing,
\begin{equation}\label{k_R as main part}
    \begin{aligned}
        |\partial_{x_1} R_n + \epr \partial_{x_1} k_{R_n} \rn| & \leq C \epr^2 \partial_{x_1} k_{R_n} \, ,
    \end{aligned}
\end{equation}
and
\begin{equation} \label{rarefaction time bound}
    \begin{aligned}
        \int_0^\infty \|\partial_{x_1} R_n\|_{L^4} ^4 \, dt\leq C \epr ^ 3 \, , \; \int_0^\infty \|\partial_{x_1} R_n\|_{L^4} ^2 \, dt\leq C \epr  \, , \\
        \int_0^\infty \|\partial_{x_1}^{2j} R_n\|_{L^2} ^2 \, dt\leq C \epr \, , \;
        \int_0^\infty \|\partial_{x_1}^2 R_n\|_{L^{1}} ^{4/3} \, dt\leq C \epr^{1/3} \, .
    \end{aligned}
\end{equation}
For convenience, we call $R_n$ the planar n-rarefaction wave from now on.

\subsection{Construction of weight functions, shift functions and the superposition wave}
We are ready to introduce the weight functions, the shift functions and the superposition wave.

Let $W_n$ be either the planar viscous n-shock wave $S_n$ or the planar n-rarefaction wave $R_n$. Recall (\ref{decomposition of S}) and (\ref{decomposition of R}).
We define the weight functions $a_{S_1}, a_{W_n}$ by
\begin{equation}\label{def of weight functions}
    \begin{aligned}
        a_{S_1} (t,x) &= 1 - \Lambda_{S_1} \eps {k_{S_1}}(t,x) \, ,\\
        a_{W_n}(t, x) &= 1 + \Lambda_{W_n} \epww  k_{W_n}(t,x)  \, ,
    \end{aligned}
\end{equation}
for large enough constants $\Lambda_{S_1},\Lambda_{W_n}>0$ that depend only on $B_j,f,\eta,\ul$. As $k_{S_1}, k_{W_n}$ are increasing, the weight function $a_{S_1}$ is decreasing and the weight function $a_{W_n}$ is increasing.
Also, we have $\|a_{S_1}\|_{C^1}, \|a_{W_n}\|_{C^1} \leq 2$ by taking $\eps, \epww$ small enough.

For any function $h:\mathbb{R}_{\geq 0}\times \mathbb{R}\times \mathbb{T}^{d-1}\rightarrow \mathbb{R}^n$ and $\mathbf{X}:\mathbb{R}_{\geq 0}\rightarrow \mathbb{R}$, we define 
\begin{align*}
    h^{\x}\bigl(t,(x_1,y)\bigl)=h\bigl(t,(x_1 + \x(t),y)\bigl) \, .
\end{align*}
We define the superposition wave 
\begin{equation} \label{eq for super wave}
    \utl (t,x) = {S}_1^{{\mathbf{X}_1}}(t,x) + W_n^{\x_n}(t,x) - \um \, ,
\end{equation}
and the weight function
\begin{align}\label{weight function}
    a(t,x) = a_{S_1}^{\x_1}(t, x) + \aw(t, x) \, ,
\end{align}
where $a_{S_1}, a_{W_n}$ are defined in (\ref{def of weight functions}), and the shift $\mathbf{X}_i$ is defined as the solution to the ODE
\begin{equation} \label{def of shift}
    \begin{cases}
        \dot{\x}_i(t) = \frac{\Tilde{C}_i}{\epss} \int a \, \eta'' (\Tilde{U})(U-\utl) \partial_{x_1} S_i^{\x_i} \, dx\, , \\
        \x_i(0)=0 \,,
    \end{cases}
\end{equation}
for a large enough constant $\Tilde{C_i}>0$ that depends only on $B_j,f,\eta,\ul$ for $i\in\{1,n\}$ if $W_n = S_n$ and for $i=1$ if $W_n=R_n$.
As the planar n-rarefaction wave $R_n$ is not shifted, we define $\x_n=0$ if $W_n=R_n$ for consistency. 
The existence and uniqueness of shifts $\x_i$ are proved in Proposition \ref{existence prop}.

By (\ref{eq for shock}) and (\ref{eq for rarefaction}), we have
\begin{equation} \label{eq for superposition wave}
    \partial_t \utl + \partial_{x_1} f(\utl) = \partial_{x_1} \bigl(B_1(\utl)\partial_{x_1} \eta'  (\utl)\bigl)  + Z + E_1 + E_2\, ,
\end{equation}
where 
\begin{equation*}
    {E_1} = \partial_{x_1} f(\utl) - \partial_{x_1} f(\so) - \partial_{x_1} f(\wn) \, ,
\end{equation*}
if ${W}_n=S_n$, then
\begin{equation*}
    \begin{aligned}
        {Z} &=\dot{\mathbf{X}}_1\partial_{x_1}\so + \dot{\mathbf{X}}_n\partial_{x_1}\st\, , \\
        {E_2} &= 
        \partial_{x_1}\bigl(B_1(\so)\partial_{x_1} \eta'(\so)\bigl) + \partial_{x_1}\bigl(B_1(\st)\partial_{x_1} \eta'(\st)\bigl) \\&\quad - \partial_{x_1} \bigl(B_1(\utl)\partial_{x_1} \eta'  (\utl)\bigl) \, , 
    \end{aligned}
\end{equation*}
and if ${W}_n={R}_n$, then
\begin{equation*}
    \begin{aligned}
        {Z} &=\dot{\mathbf{X}}_1\partial_{x_1}\so \, , \\
        {E_2} &= 
        \partial_{x_1}\bigl(B_1(\so)\partial_{x_1} \eta'(\so)\bigl)- \partial_{x_1} \bigl(B_1(\utl)\partial_{x_1} \eta'  (\utl)\bigl) \, . 
    \end{aligned}
\end{equation*}
The terms $E_1, E_2$ are error terms caused by the fact that $\utl$ is not an exact solution of the system (\ref{eq for sol}). The term $Z$ comes from the shifts. We see $Z, E_2$ depend on the choice of $W_n$, because while the planar viscous shock wave is shifted and is a solution to the viscous system (\ref{eq for shock}), the planar rarefaction wave is not shifted and is a solution to the inviscid model (\ref{eq for rarefaction}).

As $\utl$ does not depend on the transverse direction $y$, (\ref{eq for sol}) and (\ref{eq for superposition wave}) give
\begin{equation}\label{eq u and u tilde}
    \begin{aligned}
        \partial_t (U-\utl) + \partial_{x_1} \bigl(f(U)-f(\utl)\bigl) + \sum_{j=2}^{d}\partial_{x_j} \bigl(g_j(U)-g_j(\utl)\bigl) \\ = \sum_{j=1}^d \partial_{x_j}\bigl(B_j(U)\partial_{x_j} \eta'(U) -B_j(\utl)\partial_{x_j} \eta' (\utl)\bigl) - {Z}- {E_1} - {E_2} \, .
    \end{aligned}
\end{equation}
\section{Proof of Theorem \ref{thm}}\label{proof of theorem}
First, we introduce local-in-time estimates in Proposition \ref{existence prop} and a priori estimates in Proposition \ref{prop energy estimate}. Then we discuss how the two propositions prove the global existence and the asymptotic behavior results stated in Theorem \ref{thm}.
\subsection{Local-in-time estimates and a priori estimates}\label{subsection of props}
\begin{prop}\label{existence prop}
    For any $0<\epl<\epe$ and any $t_0\geq 0$, there exists $T_0>0$ that depends on $\epl,\epe$ such that the following is true. 

    \noindent
    If $\|U(t_0,\cdot)-\utl(t_0,\cdot)\|_{H^m(\mathbb{R}\times {\mathbb{T}^{d-1}})}\leq \ep_1$, then
    \begin{itemize}
        \item[1)] (\ref{eq for sol}) has a unique  solution $U$ on $[t_0,t_0+T_0]$,
        \item[2)] (\ref{def of shift}) has a unique absolutely continuous solution $\x_i$ on $[t_0,t_0+T_0]$,
        \item[3)] $U-\utl \in C\bigl([t_0, t_0+T_0]; H^m(\mathbb{R}\times {\mathbb{T}^{d-1}})\bigl)$,
        \item[4)] $\|U(t,\cdot)-\utl(t,\cdot)\|_{H^m(\mathbb{R}\times {\mathbb{T}^{d-1}})}\leq \epe$ for any $t_0\leq t\leq t_0+T_0$.
    \end{itemize}
\end{prop}
The local-in-time existence and uniqueness of the solution $U$ can be done similarly as in Serre's paper \cite{Slocalexstc}. Since the viscosity matrices of our system (\ref{eq for sol}) are positive definite, the proof will be simpler. 
The existence and uniqueness of shifts $\x_i$ can be shown in the same way as in subsection 3.3 of \cite{MVWns}.

Before introducing a priori estimates, we first establish the assumption of a priori estimates. Since sections $\ref{section energy estimates}$ and $\ref{H^m contraction}$ give the proof of a priori estimates, the following assumption will be the assumption of all lemmas in both sections.
\begin{assum}\label{assumption}
    Let $U$ be solution to (\ref{eq for sol}) on $[0,T]$ for some $T>0$.
    Let $\utl$ be the superposition wave defined in (\ref{eq for super wave}) with absolutely continuous shifts $\x_1,\x_n$ defined in (\ref{def of shift}) and weight function $a$ defined in (\ref{weight function}).
    Assume $\eps, \epww< \epw$,
    \begin{equation*}
        U - \utl \in C\bigl([0,T]; H^m(\mathbb{R}\times{\mathbb{T}^{d-1}})\bigl) \, ,
    \end{equation*}
    and
    \begin{equation} \label{L infinity bound}
        \| U-\utl\|_{L^\infty (0, T ; \, H^m (\mathbb{R}\times{\mathbb{T}^{d-1}}))} \leq \epe \, .
    \end{equation}
\end{assum}

By Sobolev embedding, (\ref{L infinity bound}) in Assumption \ref{assumption} implies
\begin{equation}\label{L infinity}
    \|U - \utl\|_{L^\infty((0,T)\times (\mathbb{R}\times {\mathbb{T}^{d-1}}))} \leq C \epe \, .
\end{equation}

\begin{prop} \label{prop energy estimate}
    For any $\ul \in \mathbb{R}^n$, there exist $\epw, \epe, C_0, \Lambda_{S_1}, \Lambda_{W_n}, \Tilde{C}_{1}, \Tilde{C}_{n}>0$ such that the following is true.
    
    Assume Assumption \ref{assumption}. Then
    \begin{align}\label{prop energy estimate inequality 1}
        \begin{split}
            \underset{t\in [0,T]}{\textup{sup}}\|U(t,\cdot)-\utl(t,\cdot)\|_{H^m(\mathbb{R}\times{\mathbb{T}^{d-1}})} \\
            + \sqrt{ \int_0^T \sum_{k=0}^m D_k(U) + G_{S_1}(U) + G_{W_n}(U) + Y dt }
            \\ \leq  C_0 \|U_0-\utl_0\|_{H^m(\mathbb{R}\times{\mathbb{T}^{d-1}})} + C_0 E \, ,
        \end{split}
     \end{align}
     and for any $\beta\in \mathbb{N}^d$ such that $1\leq |\beta|\leq m$,
     \begin{equation}\label{estimate for time asymptotic}
        \begin{aligned}
            &\int_0^T \Big|\frac{d}{dt} \int \big|\partial_x^\beta \bigl(U(t,\cdot)-\utl(t,\cdot)\bigl)\big|^2\, dx\Big | \, dt
            \\&\leq C_0 \int_0^T \Bigl(\sum_{k=0}^m D_k(U) + G_{S_1}(U) + G_{W_n}(U) + Y \Bigl) \, dt + C_0 E^2 \, ,
        \end{aligned}
     \end{equation}
    where
    \begin{equation}\label{def of D, D1, G_S, G_R}
        \begin{aligned}
            D_k(U)&= \sum_{\alpha\in \mathbb{N}^d, |\alpha|=k+1}\int \big|\partial_{x}^{\alpha} (U-\utl)\big|^2 \, dx \, , \\
            G_{S_1}(U)&=\int |U-\utl|^2 |\partial_{x_1} \so| \, dx\, , \\
            G_{W_n}(U)&=\int |U-\utl|^2 |\partial_{x_1} \wn| \, dx \, ,\\
            Y&= 
            \begin{cases}
                \eps|\dot{\x}_1|^2 + \epww |\dot{\x}_n|^2 \, , \textup{ if } W_n = S_n \, , \\
                \eps|\dot{\x}_1|^2 \, , \textup{ if } W_n =R_n \, ,
            \end{cases}
            \\
            E&=\begin{cases}
                \epws + \eps \, , \textup{ if } W_n = S_n \, , \\
                \epr ^ {1/6}\, , \textup{ if } W_n =R_n \, .
            \end{cases}
        \end{aligned}
    \end{equation}
    In addition, 
    \begin{equation}
        |\dot{\x}_1(t)| + |\dot{\x}_n(t)| \leq C_0 \| U(t,\cdot)-\utl(t,\cdot)\|_{L^\infty (\mathbb{R}\times {\mathbb{T}^{d-1}})}\, , \quad \forall \, t\in [0,T] \, .
    \end{equation}
\end{prop}
\subsection{Global existence and estimates}\label{Global existence and estimate}
We define
\begin{equation*}
    \p(t)=\|U(t,\cdot)-\utl(t,\cdot)\|_{H^m(\mathbb{R}\times {\mathbb{T}^{d-1}})} \, .
\end{equation*}
We fix $\epe, C_0$ as in Proposition \ref{prop energy estimate}. We take  
the strength of the initial perturbation $\ept$ and the wave strength $\epw$ in Theorem \ref{thm} and $\ep_1$ in Proposition \ref{existence prop} to be small enough such that 
\begin{equation}\label{relationship ep}
    0<C_0\p(0)+C_0 E <\epl<\frac{\epe}{2} \, .
\end{equation}
We define
\begin{align*}
    T = \textup{sup}\Bigl\{t\geq 0: U-\utl \in C\bigl([0,T]; H^m(\mathbb{R}\times{\mathbb{T}^{d-1}})\bigl) \, , \; \p(t)<\epe\Bigl\} \, .
\end{align*}
Assume $T< \infty$. Then Proposition $\ref{prop energy estimate}$ gives
\begin{align*}
    \p(T)\leq C_0\p(0)+C_0 E <\epl \, .
\end{align*}
By Proposition \ref{existence prop}, there exists $T_0>0$ such that 
\begin{align*}
     U-\utl \in C\bigl([T,T+T_0]; H^m(\mathbb{R}\times{\mathbb{T}^{d-1}})\bigl) \, \textup{ and } \, \p(t) \leq \frac{\epe}{2} \textup{ for any } t\in [T, T+T_0]  \, .
\end{align*}
Contradiction! Hence, we get $T=\infty$ and (\ref{thm continuity}) in Theorem \ref{thm}. Now we can apply Proposition \ref{prop energy estimate} on $[0,\infty)$ and get
\begin{align}\label{global estimate energy prop}
    \begin{split}
        \underset{t\in [0,\infty)}{\textup{sup}}\|U-\utl\|_{H^m(\mathbb{R}\times {\mathbb{T}^{d-1}})} \\
        + \sqrt{ \int_0^\infty \sum_{k=0}^m D_k(U) + G_{S_1}(U) + G_{W_n}(U) + Y dt }
        \\ \leq  C_0 \|U_0-\utl_0\|_{H^m(\mathbb{R}\times {\mathbb{T}^{d-1}})} + C_0 E <\infty \, ,
    \end{split}
\end{align}
and for any $\beta = \mathbb{N}^d$ such that $1\leq |\beta| \leq m$,
\begin{equation}\label{estimate for time asymptotic infnity}
        \begin{aligned}
            &\int_0^\infty \Big|\frac{d}{dt} \int \big|\partial_x^\beta \bigl(U(t,\cdot)-\utl(t,\cdot)\bigl)\big|^2\, dx\Big | \, dt
            \\&\leq C_0 \int_0^\infty \Bigl(\sum_{k=0}^m D_k(U) + G_{S_1}(U) + G_{W_n}(U) + Y \Bigl) \, dt + C_0 E^2 < \infty\, .
        \end{aligned}
\end{equation}
In addition,
\begin{equation}\label{global estimate energy prop shift}
    |\dot{\x}_1(t)| + |\dot{\x}_n(t)| \leq C_0 \| U(t,\cdot)-\utl(t,\cdot)\|_{L^\infty (\mathbb{R}\times {\mathbb{T}^{d-1}})}\, , \quad \forall \, t\geq 0 \, .
\end{equation}
By (\ref{global estimate energy prop}), we get (\ref{thm second deri}) in Theorem \ref{thm}.
\subsection{Time-asymptotic behavior}\label{time asymptotic behavior}
Let $\beta\in \mathbb{N}^d$ be such that $1\leq |\beta|\leq m$. We define 
\begin{equation*}
    g(t) = \big\|\partial_x^{\beta} \bigl(U(t,\cdot)-\Tilde{U}(t,\cdot)\bigl)\big\|_{L^2(\mathbb{R}\times {\mathbb{T}^{d-1}})}^2 \, .
\end{equation*}
We show the classical estimate
\begin{equation} \label{classical estimate}
    \int_0^{\infty} |g(t)| + |{g}'(t)| \, dt < \infty \,.
\end{equation}
By (\ref{global estimate energy prop}), we get
\begin{align*}
    \int_0^{\infty} |g(t)| \, dt \leq \int_0^{\infty} \sum_{k=0}^{m-1}D_{k}(U) \, dt < \infty \, .
\end{align*}
By (\ref{estimate for time asymptotic infnity}), we have
\begin{align*}
     \int_0^\infty |g'(t)|\, dt = \int_0^\infty \Big|\frac{d}{dt} \int \big|\partial_x^\beta \bigl(U(t,\cdot)-\utl(t,\cdot)\bigl)\big|^2\, dx\Big |  \, dt < \infty.
\end{align*}
The classical estimate (\ref{classical estimate}) gives
\begin{equation*}
    \underset{t\rightarrow \infty}{\textup{lim}}\big\|\partial_x^{\beta} \bigl(U(t,\cdot) - \utl(t,\cdot)\bigl)\big\|_{L^2(\mathbb{R}\times {\mathbb{T}^{d-1}})} = 0 \, .
\end{equation*}
The Gagliardo–Nirenberg inequality proved in \cite{Ninequality}, the periodicity in the transverse direction, and (\ref{global estimate energy prop}) give
\begin{equation}\label{time asymptotic L infinity}
    \underset{t\rightarrow \infty}{\textup{lim}}\|U(t,\cdot) - \utl(t,\cdot)\|_{L^\infty(\mathbb{R}\times {\mathbb{T}^{d-1}})} = 0 \, .
\end{equation}
By (\ref{time asymptotic L infinity}) and Lemma \ref{property of rarefaction}, we get (\ref{thm pointwise convergence}) in Theorem \ref{thm}. 
By (\ref{global estimate energy prop shift}) and (\ref{time asymptotic L infinity}), we get (\ref{thm X shift}) in Theorem \ref{thm}.

\section{Energy estimate}\label{section energy estimates}
We show the $L^2$ estimates by the $a$-contraction method with shifts in this section. Later in section \ref{H^m contraction}, induction will help us get the $H^m$ estimates and finish the proof of Proposition \ref{prop energy estimate}.

First, we develop tools to handle interaction terms and higher-order terms in subsections \ref{subsec wave interaction} and \ref{subsec rearrange}. Then we apply the $a$-contraction method with shifts. The hyperbolic ``scalarization" is discussed in subsection \ref{subsec relative entropy}. The positive hyperbolic remainder corresponding to the special direction of the planar viscous shock wave motivates the Poincar\'{e} type inequality introduced in subsection \ref{subsec poincare}. Finally in subsection \ref{subsec L^2 contraction}, we choose the constants $\Lambda_{S_1}, \Lambda_{W_n}, \Tilde{C}_{1}, \Tilde{C}_{n}$ defining weight functions and shift functions in a way that makes both hyperbolic ``scalarization" and Poincar\'{e} type inequality work.
\subsection{Wave interaction estimates}\label{subsec wave interaction}
To control the interaction between waves, the idea is to take the shifts small enough that the main layer regions do not overlap.
\begin{lemm} \label{interaction bound}
    Let $\x_1, \x_n$ be the shifts defined in (\ref{def of shift}). Assume Assumption \ref{assumption}. Then for any $0\leq t\leq T$,
    \begin{align*}
        &\big\||\partial_{x_1} \so| |R_n-\um|\big\|_{L^2} + 
        \big\||\partial_{x_1} R_n| |\so-\um|\big\|_{L^2} +
        \big\||\partial_{x_1} R_n| | \partial_{x_1} \so|\big\|_{L^2} \\
        &\leq C \epr \eps e^{-C\eps t} \, , \quad \textup{ if } W_n = R_n \, ,\\
        &\big\| |\partial_{x_1} \so| |\st-\um|^2 \big\|_{L^1} + \big\||\partial_{x_1} \st||\so -\um|^2\big\|_{L^1}  + \big\||\partial_{x_1} \st| | \partial_{x_1} \so|\big\|_{L^1} \\
        &\leq C (\epws^2 \eps e^{-C\eps t} + \epws \eps^2 e^{-C\epws t}) \, , \quad \textup{ if } W_n = S_n \, .
    \end{align*}
\end{lemm}
\begin{proof} 
First, we define the spectral gap 
\begin{align*}
    \Delta = \lambda_n(\ul) - \lambda_1(\ul) \, .
\end{align*}
By (\ref{def of shift}) and (\ref{L infinity}), we take $\epe$ small enough such that for any $i\in \{1,n\}$,
\begin{equation*}
    |\x_i(t)| \leq C  \Tilde{C}_{i} \epe t < \frac{\Delta t}{8} \, .
\end{equation*}
For any $x_1 > \frac{(\lambda_1(\ul)+\lambda_n(\ul))t}{2}$, we have
\begin{align*}
    x_1-\sio t+\x_1(t) >  \frac{\Delta t}{4} \, . 
\end{align*}
Lemma \ref{pointwise bound of shock} implies that for any $x_1> \frac{(\lambda_1(\ul)+\lambda_n(\ul))t}{2}$,
\begin{align*}
    |\so - \um| &\leq C \eps e^{-C\eps |x_1-\sio t +\mathbf{X}_1(t)|} \\
    & \leq C \eps \textup{exp}\bigl(\frac{-C\eps|x_1-\sio t +\mathbf{X}_1(t)|}{2}\bigl) \textup{exp}\bigl(\frac{-C\eps\Delta t}{8}\bigl) \, , \\
    |\partial_{x_1} \so| &\leq C \eps^2 e^{-C\eps |x_1-\sio t+\mathbf{X}_1(t)|} \\
    & \leq C \eps^2 \textup{exp}\bigl(\frac{-C\eps|x_1-\sio t+\mathbf{X}_1(t)|}{2}\bigl) \textup{exp}\bigl(\frac{-C\eps\Delta t}{8}\bigl) \, .
\end{align*}
For any $x_1 \leq \frac{(\lambda_1(\ul)+\lambda_n(\ul))t}{2}$, we have
\begin{align*}
    x_1 -\lambda_n(\um) t &\leq  -\frac{\Delta t}{4} \, , \\
    x_1 -\sit t + \x_n(t) &\leq  -\frac{\Delta t}{4} \, .
\end{align*}
Lemma \ref{property of rarefaction} implies that for any $x_1\leq \frac{(\lambda_1(\ul)+\lambda_n(\ul))t}{2}$, 
\begin{align*}
    |R_n-\um| \, ,\; |\partial_{x_1} R_n| &\leq C \epr e^{-2|x_1-\lambda_n(\um)t|} \\
    & \leq C \epr \textup{exp}\bigl(-|x_1-\lambda_n(\um)t|\bigl) \textup{exp}\bigl(-\frac{\Delta t}{4}\bigl) \, ,
\end{align*}
and Lemma \ref{pointwise bound of shock} implies that for any $x_1\leq \frac{(\lambda_1(\ul)+\lambda_n(\ul))t}{2}$,
\begin{align*}
    |\st - \um| &\leq C \epws e^{-C\epws |x_1 -\sit t + \x_n(t)|} \\
    & \leq C \epws \textup{exp}\bigl(\frac{-C\epws|x_1 -\sit t + \x_n(t)|}{2}\bigl) \textup{exp}\bigl(\frac{-C\epws\Delta t}{8}\bigl) \, , \\
    |\partial_{x_1} \st| &\leq C \epws^2 e^{-C\epws |x_1 -\sit t + \x_n(t)|} \\
    & \leq C \epws^2 \textup{exp}\bigl(\frac{-C\epws|x_1 -\sit t + \x_n(t)|}{2}\bigl) \textup{exp}\bigl(\frac{-C\epws\Delta t}{8}\bigl) \, .
\end{align*}
We have
\begin{align*}
    &|\partial_{x_1} \so| \bigl(|R_n-\um|+|\partial_{x_1} R_n|\bigl) \\ 
    &\leq 
    \begin{cases}
        C \epr \eps^2 e^{-C\eps |x_1-\sio t+\mathbf{X}_1(t)|} e^{-C\eps t}\, , \; \textup{ if } x_1 > \frac{(\lambda_1(\ul)+\lambda_n(\ul))t}{2} \, , \\
        C \epr \eps^2 e^{-|x_1-\lambda_n(\um)t|} e^{-Ct}\, , \; \textup{ if } x_1 \leq \frac{(\lambda_1(\ul)+\lambda_n(\ul))t}{2} \, , 
    \end{cases} \\
    &|\partial_{x_1} R_n||\so-\um| \\ &\leq 
    \begin{cases}
        C |\partial_{x_1} R_n| \eps e^{-C\eps |{x_1}-\sio t+\mathbf{X}_1(t)|} e^{-C\eps t}\, , \; \textup{ if } {x_1} > \frac{(\lambda_1(\ul)+\lambda_n(\ul))t}{2} \, , \\
        C \epr \eps e^{-|{x_1}-\lambda_n(\um)t|} e^{-Ct}\, , \; \textup{ if } {x_1} \leq \frac{(\lambda_1(\ul)+\lambda_n(\ul))t}{2} \, . 
    \end{cases} 
\end{align*}
Hence, we get
\begin{align*}
    &\big\| |\partial_{x_1} \so| |R_n-\um|\big\|_{L_{x_1}^2}^2 + \big\||\partial_{x_1} R_n| | \partial_{x_1} \so|\big\|_{L_{x_1}^2}^2 \\ 
    &\leq C \epr^2 \eps^3 e^{-C\eps t} \int \eps (e^{-C\eps |{x_1}|} + e^{-|{x_1}|}) \, d{x_1} \leq C \epr ^2 \eps ^ 3 e^{-C\eps t} \, , \\
    &\big\||\partial_{x_1} R_n||\so -\um|\big\|_{L_{x_1}^2}^2  \\
    &\leq C \epr \eps^2 e^{-C\eps t} \int |\partial_{x_1} R_n| \, d{x_1} + 
    C \epr^2 \eps^2 e^{-Ct} \int e^{-|{x_1}|} \, d{x_1} \\
    &\leq C \epr^2 \eps^2 e^{-C\eps t} \, .
\end{align*}
We have
\begin{align*}
    &|\partial_{x_1} \so| |\st-\um|^2 + |\partial_{x_1} \st||\so-\um|^2 + |\partial_{x_1} \st| | \partial_{x_1} \so|\\ 
    &\leq 
    \begin{cases}
        C \epws^2 \eps^2 e^{-C\eps |x_1-\sio t+\mathbf{X}_1(t)|} e^{-C\eps t}\, , \; \textup{ if } x_1 > \frac{(\lambda_1(\ul)+\lambda_n(\ul))t}{2} \, , \\
        C \epws^2 \eps^2 e^{-C\epws|x_1 -\sit t + \x_n(t)|} e^{-C\epws t}\, , \; \textup{ if } x_1 \leq \frac{(\lambda_1(\ul)+\lambda_n(\ul))t}{2} \, .
    \end{cases} 
\end{align*}
Hence, we get
\begin{align*}
    &\big\| |\partial_{x_1} \so| |\st-\um|^2\big\|_{L_{x_1}^1} + \big\||\partial_{x_1} \st||\so -\um|^2\big\|_{L_{x_1}^1} + \big\||\partial_{x_1} \st| | \partial_{x_1} \so|\big\|_{L_{x_1}^1}\\
    &\leq C \epws ^2 \eps e^{-C\eps t} \int \eps e^{-C\eps |{x_1}|} \, d{x_1} + C \epws  \eps^2 e^{-C\epws t} \int \epws e^{-C\epws |{x_1}|} \, d{x_1}\\
    &\leq C (\epws ^2 \eps e^{-C\eps t} + \epws  \eps^2 e^{-C\epws t}) \, .
\end{align*}
\end{proof}
\subsection{Higher derivatives estimates}\label{subsec rearrange}
Let 
\begin{equation}\label{notataion of perturbation}
    \psi = U-\utl \, .
\end{equation}
For any $k\in \mathbb{N}^*$, we define
\begin{align}\label{def of L^k}
    \mathcal{L}^k = \Bigl\{ \Pi_{j=1}^l |\partial_x^{\beta_j}\psi| : l\in \mathbb{N}^* \,, \;\beta_j\in \mathbb{N}^d\, , \; 1\leq |\beta_j|\leq k \, , \; \sum_{j=1}^l |\beta_j| \leq k+1 \Bigl\} \, .
\end{align}
Lemma \ref{lemm reorganization} plays an important role in section \ref{H^m contraction} when working with higher derivatives. It singles out terms that need to be handled differently and unifies the rest in the same form. The last two inequalities in the lemma will be used in the proof of Lemma \ref{lemma rem} to evaluate the interaction terms induced by $E_1, E_2$.
\begin{lemm}\label{lemm reorganization}
    Assume Assumption \ref{assumption}.
    Let $1\leq k\leq m$ where $m$ is defined in (\ref{value of m}). Let $M:\mathbb{R}^n \rightarrow \mathbb{R}^{n\times n}$ and $F:\mathbb{R}^n \rightarrow \mathbb{R}^{n}$ be smooth functions. 
    Then there exists a constant $C>0$ such that for any $\alpha_{k}, \alpha_1 \in \mathbb{N}^d$ such that $|\alpha_{k}|=k \, , \;  |\alpha_1|=1$,
    \begin{align*}
        &\Big|\partial_x^{\alpha_{k}} \bigl(M(U)\partial_x^{\alpha_1} U -M(\utl)\partial_x^{\alpha_1} \utl\bigl)-\bigl(M(U)\partial_x^{\alpha_{k} +\alpha_1} U -M(\utl)\partial_x^{\alpha_{k} +\alpha_1} \utl\bigl)\Big| \\
        &\leq C \bigl(\sum_{L\in\mathcal{L}^{k}} L + |\psi||\partial_{x_1} \utl|\bigl) \, , \\
        &\Big|\partial_x^{\alpha_{k}} \bigl(F(U) -F(\utl)\bigl)\Big| \leq C \bigl(\sum_{L\in\mathcal{L}^{k}} L + |\psi||\partial_{x_1} \utl|\bigl) \, , \\
        &\Big|\partial_x^{\alpha_{k}} \bigl(M(\utl)\partial_{x_1} \utl -M(\so)\partial_{x_1} \so\bigl)\Big|\\
        &\leq C \bigl(|\partial_{x_1}^{k + 1} \wn| + |\partial_{x_1} \wn|^2 + |\partial_{x_1}\wn||\partial_{x_1}\so|+ |\wn-\um| |\partial_{x_1}\so|\bigl) \, , \\
        &\Big|\partial_x^{\alpha_{k}} \bigl(F(\utl)-F(\so)-F(\wn)\bigl)\Big| \\
        &+ \Big|\partial_x^{\alpha_{k}} \bigl(M(\utl)\partial_{x_1} \utl -M(\so)\partial_{x_1} \so - M(\wn)\partial_{x_1} \wn\bigl)\Big|\\
        &\leq C \bigl(|\partial_{x_1} \wn||\partial_{x_1} \so| + |\wn-\um||\partial_{x_1} \so| + |\so-\um||\partial_{x_1} \wn|\bigl) \, .
    \end{align*}
\end{lemm}
\begin{proof} 
For any $k_1,k_2\in \mathbb{N}^*$ and any function $V,\Tilde{V}:\mathbb{R}_{\geq 0}\times \mathbb{R}\times\mathbb{T}^{d-1} \rightarrow \mathbb{R}^n$, we define
\begin{align*}
        \mathcal{L}^{k_1,k_2}(V,\Tilde{V}) &= \Bigl\{\Pi_{j=1}^l |\partial_x^{\beta_j} V| \, \Pi_{j=1}^{\Tilde{l}} |\partial_x^{\Tilde{\beta_j}} \Tilde{V}| : l\in \mathbb{N}^*, \, \Tilde{l}\in \mathbb{N},\,  \beta_j,\Tilde{\beta_j} \in \mathbb{N}^d ,\, \\& \quad \quad 1\leq |\beta_j|, |\Tilde{\beta_j}|\leq k_1, \,\sum_{j=1}^l |\beta_j|+\sum_{j=1}^{\Tilde{l}}|\Tilde{\beta_j}| = k_2 \, , \; \sum_{j=1}^l |\beta_j| \geq 1 \Bigl\}\, .
    \end{align*}
We show the first inequality. If we apply the chain rule to
\begin{equation*}
    \partial_{x}^{\alpha_{k}} \bigl(M(\cdot)  \partial_x^{\alpha_1} (\cdot)\bigl)  \, ,
\end{equation*}
then either all $\partial_{x}^{\alpha_{k}}$ fall on $\partial_x^{\alpha_1} (\cdot)$ or some fall on $M(\cdot)$. Therefore, we know that
\begin{align*}
    \Big|\partial_x^{\alpha_{k}} \bigl(M(U)\partial_x^{\alpha_1} U -M(\utl)\partial_x^{\alpha_1} \utl\bigl)-\bigl(M(U)\partial_x^{\alpha_{k}+\alpha_1}\partial_x U -M(\utl)\partial_x^{\alpha_{k}+\alpha_1} \utl\bigl)\Big| 
\end{align*}
is bounded by the sum of the absolute value of the terms in the following form
\begin{align*}
    \mathbf{M}:=\bigl(M^{\beta_j}(U)\otimes_{j=1}^{l-1} \partial_x^{\beta_j} U \bigl)\partial_x^{\beta_l} U- \bigl(M^{\beta_j}(\utl)\otimes_{j=1}^{l-1} \partial_x^{\beta_j} \utl \bigl)\partial_x^{\beta_l} \utl \, ,
\end{align*}
where $l\geq 2$, $\beta_{j}\in\mathbb{N}^d$, $1\leq |\beta_j|\leq k$, $\sum_{j=1}^{l}|\beta_j|=k+1$, and $M^{\beta_{j}}$ is some derivative of $M$. By (\ref{L infinity}), Lemma \ref{pointwise bound of shock}, and Lemma \ref{property of rarefaction}, we have
\begin{align*}
    |\mathbf{M}| &= \Big| \bigl(M^{\beta_j}(U)\otimes_{j=1}^{l-1} \partial_x^{\beta_j} (\psi + \utl) \bigl)\partial_x^{\beta_l} (\psi +\utl) - \bigl(M^{\beta_j}(\utl)\otimes_{j=1}^{l-1} \partial_x^{\beta_j} \utl \bigl)\partial_x^{\beta_l} \utl \Big| \\
    &\leq C\sum_{\mathcal{L}^{k, k+1}(\psi, \utl)}L + \Big|\Bigl( \bigl(M^{\beta_j}(U) - M^{\beta_j}(\utl)\bigl)\otimes_{j=1}^{l-1} \partial_x^{\beta_j} \utl \Bigl)\partial_x^{\beta_l} \utl\Big| \\
    &\leq C \sum_{\mathcal{L}^{k}}L + C |\psi| |\partial_{x_1}\utl|
    \, .
\end{align*}
Note the constants $C$ in this proof depend on $B_j,f,\eta,\ul,k$, but the dependency on $k$ does not matter as $m$ is fixed and finite.

Similarly, we can show
\begin{align*}
    \Big|\partial_x^{\alpha_k} \bigl(F(U) -F(\utl)\bigl)\Big| \leq C \bigl(\sum_{\mathcal{L}^{k, k}(\psi, \utl)}L + |\psi| |\partial_{x_1}\utl|\bigl) 
    \leq C\bigl(\sum_{\mathcal{L}^{k}}L + |\psi| |\partial_{x_1}\utl|\bigl) \, ,
\end{align*}
and
\begin{align*}
    &\Big|\partial_x^{\alpha_k} \bigl(M(\utl)\partial_x \utl -M(\so)\partial_x \so\bigl)\Big|\\
    &\leq C \bigl(\sum_{\mathcal{L}^{k+1, k+1}(\wn,\,  \so)}L + |\wn-\um| |\partial_{x_1}\so|\bigl) \\
    &\leq C \bigl(|\partial_{x_1}^{k + 1} \wn| + |\partial_{x_1} \wn|^2 + |\partial_{x_1}\wn||\partial_{x_1}\so|+ |\wn-\um| |\partial_{x_1}\so|\bigl)
    \, .
\end{align*}
Finally, we show the last inequality. We know
\begin{align*}
    \Big|\partial_x^{\alpha_k} \bigl(F(\utl)-F(\so)-F(\wn)\bigl)\Big|
\end{align*}
is bounded by the sum of the absolute value of the terms in the following form
\begin{align*}
    \mathbf{F}:=&F^{\beta_j}(\utl)\otimes_{j=1}^{l} \partial_x^{\beta_j} \utl  - F^{\beta_j}(\so)\otimes_{j=1}^{l} \partial_x^{\beta_j} \so - F^{\beta_j}(\wn)\otimes_{j=1}^{l} \partial_x^{\beta_j} \wn \, ,
\end{align*}
where $l\geq 1$, $\beta_{j}\in\mathbb{N}^d$, $1\leq |\beta_j|\leq k$, $\sum_{j=1}^{l}|\beta_j|=k$, and $F^{\beta_{j}}$ is some derivative of $F$. By Lemma \ref{pointwise bound of shock} and Lemma \ref{property of rarefaction}, we have
\begin{align*}
    |\mathbf{F}| &= \Big|F^{\beta_j}(\utl)\otimes_{j=1}^{l} \partial_x^{\beta_j} (\so+\wn) - F^{\beta_j}(\so)\otimes_{j=1}^{l} \partial_x^{\beta_j} \so 
    \\&\quad - F^{\beta_j}(\wn)\otimes_{j=1}^{l} \partial_x^{\beta_j} \wn \Big| \\
    &\leq C \bigl(|\partial_{x_1} \wn||\partial_{x_1} \so| + |\wn-\um||\partial_{x_1}\so| + |\so-\um||\partial_{x_1} \wn|\bigl)\, .
\end{align*}
Similarly, we can bound
\begin{align*}
    \Big|\partial_x^{\alpha_{k}} \bigl(M(\utl)\partial_{x_1} \utl -M(\so)\partial_{x_1} \so - M(\wn)\partial_{x_1} \wn\bigl)\Big| \, .
\end{align*}
\end{proof}

\subsection{Relative entropy method}\label{subsec relative entropy}
For each layer corresponding to $\so$ and $\wn$, we will construct a basis of the phase space that is well-adapted to both the special direction of the wave $\ro$ (respectively $\rn$) and the dissipation matrix $B_1(\ul)$. As 
the dissipation takes place in the so-called entropic variables $\eta'(U)-\eta'(\utl)$, we project such quantity onto the bases. Since $U-\utl$ is small, we know
\begin{align*}
    \eta'(U)-\eta'(\utl)\approx \eta''(\ul)(U-\utl) \, .
\end{align*}
Hence, the corresponding natural special direction of the wave for the entropic variables is $\eta''(\ul)\ro=\lo$ (respectively $\eta''(\ul)\rn=\lnn$). The point is to work with a basis that is orthogonal with respect to the dissipation matrix and contains $\lo$ (respectively $\lnn$).

When working with the layer of $\so$, we complete $\lo$ into an orthogonal basis of $\mathbb{R}^n$ with respect to the dissipation. In particular, we choose a basis $(\mathbf{v_1^{(1)}} = \lo,\mathbf{v_2^{(1)}},...,\mathbf{v_n^{(1)}})$ such that for any $2\leq i\leq n$ and any $1\leq j\leq n$,
\begin{align*}
    \Tilde{B}_1(\ul) \voi \cdot \mathbf{v_j^{(1)}} = 0 \, , \; B_1(\ul)\voi \cdot  {\voi} =1 \, , 
\end{align*}
where $\Tilde{B}_1 = B_1+B_1^T$. This is always possible thanks to the Gram-Schmidt process. We project the perturbation in entropic variables onto this basis:
\begin{equation} \label{decomposition of S_1}
    \eta'(U)(t,x) - \eta'(\utl)(t,x) = \mu_1(t,x) \lo + \sum_{i=2}^n \mu_i(t,x) \voi \,.
\end{equation}
When working with the hyperbolic terms, we will work with the conserved quantity $U-\utl$. By Taylor expansion and (\ref{L infinity}), we have
\begin{align}\label{decomposition of S_1 taylor}
    \big|(U-\utl)-\bigl(\mu_1 \ro + \sum_{i=2}^n \mu_i \, \eta''(\ul)^{-1}\voi\bigl)\big| \leq C (\epe + \epw)|U-\utl| \, .
\end{align}
As $(\ro, \eta''(\ul)^{-1}\mathbf{v_2^{(1)}}, ..., \eta''(\ul)^{-1}\mathbf{v_n^{(1)}})$ is 
a basis, we have
\begin{align}\label{property of perturbation 1}
    c \sum_{i=2}^n |\mu_i|^2 \leq \bigl|\sum_{i=2}^n \mu_i \, \eta''(\ul)^{-1}\voi\bigl|^2 \leq C \sum_{i=2}^n |\mu_i|^2 \, ,
\end{align}
for some constants $c,C>0$ that depend only on $B_j,f,\eta,\ul$.

Let $P$ be the projection onto $\textup{span}\{\mathbf{r_2},...,\rn\}$, i.e., 
\begin{align}\label{def of projection}
    P(v) = v - (v\cdot \ro) \ro \textup{ for any } v\in \mathbb{R}^n \, .
\end{align}
For any $2\leq i\leq n$,
\begin{align*}
    \eta''(\ul)^{-1} \mathbf{v_{i}^{(1)}} \in \bigl\{v:  \Tilde{B}_1(\ul)\eta''(\ul) \ro \cdot \eta''(\ul) v = 0\bigl\} =: V \, .
\end{align*}
As $\ro\notin V$, there exists $C>0$ such that 
\begin{align}\label{property of perturbation 2}
    |P(v)|^2 \geq C |v|^2 \, \textup{ for any } v\in V\, .
\end{align}
When working with the layer of $\wn$, we complete $\lnn$ into an orthogonal basis of $\mathbb{R}^n$ with respect to the dissipation. In particular, we choose a basis $(\mathbf{v_1^{(n)}},...,\mathbf{v_{n-1}^{(n)}},\mathbf{v_n^{(n)}}=\lnn)$ such that for any $1\leq i\leq n-1$ and any $1\leq j\leq n $,
\begin{align*}
    \Tilde{B}_1(\ul) \vti \cdot \mathbf{v_j^{(n)}} = 0 \, , \; B_1(\ul)\vti \cdot  {\vti} =1 \, .
\end{align*}
We project the perturbation in entropic variables onto this basis:
\begin{equation} \label{decomposition of W}
    \eta'(U)(t,x) - \eta'(\utl)(t,x) = \nu_n(t,x) {\lnn} + \sum_{i=1}^{n-1} \nu_i(t,x) \vti \,.
\end{equation}
Such projection has similar properties as (\ref{decomposition of S_1}).

We choose such bases to make the best use of dissipation in the special directions of the planar shock waves. In the case that $W_n$ is a planar rarefaction wave, we could work with the natural hyperbolic basis $(\ro,...,\rn)$ since we do not need the Poincar\'{e} inequality. However, for the sake of consistency, we will use the basis $(\mathbf{v_1^{(n)}},...,\mathbf{v_{n-1}^{(n)}},\lnn)$ in the case $W_n=R_n$ too. 

\noindent
\textbf{Relative functions}.
The relative flux $(f,g_2,...,g_n)$ is defined by
\begin{equation}\label{eq for relative flux}
    \begin{aligned}
        &f(U|V) = f(U) - f(V) - f'(V)(U-V) \, , \\
        &g_j(U|V) = g_j(U) - g_j(V) - g_j'(V)(U-V) \, ,\textup{ for } 2\leq j\leq d\, .
    \end{aligned} 
\end{equation}
The flux of the relative entropy $q=(q_1,...,q_d)$ is defined by 
\begin{equation} \label{eq for flux of RE}
    q_j(U;V) = 
    \begin{cases}
        q_1(U) - q_1(V) - \eta'(V)\bigl(f(U) - f(V)\bigl) \, , \textup{ if } j=1 \, , \\
        q_j(U) - q_j(V) - \eta'(V)\bigl(g_j(U) - g_j(V)\bigl) \, , \textup{ if } 2\leq j\leq d \, ,
    \end{cases}
\end{equation}
where $q_j(\cdot)$, the entropy flux of $\eta$, is defined in (\ref{def of entropy flux}).

We want to study the evolution of the weighted relative entropy
\begin{equation*}
    \int a(t,x) \, \eta\bigl(U(t,x)|\Tilde{U}(t, x)\bigl) \,dx \,.
\end{equation*}
It involves controlling layer quantities of the form:
\begin{align*}
    \int |\partial_{x_1}\so| F_1 \, dx \, , \; \int |\partial_{x_1} \wn| F_n \, dx \, .
\end{align*}
In each of these layers, a ``scalarization" effect takes place. Such effect damps the perturbation $U-\utl$ in the transverse direction well-adapted to both the special direction of the wave $\ro$ (respectively $\rn$) and the diffusion eigen-direction. 

The main lemma of this section is the following. We postpone our choice of $\Lambda_{S_1}, \Lambda_{W_n},\Tilde{C}_1,\Tilde{C}_n$ to subsection \ref{subsec L^2 contraction}, where we unify the choice of constants for both hyperbolic ``scalarization" and Poincar\'{e} inequality.
\begin{lemm} \label{lemma rem} 
    For any $\ul \in \mathbb{R}^n$, any $0<\de<1$, and any $\Lambda_{S_1}, \Lambda_{W_n}, \Tilde{C}_1, \Tilde{C}_n>0$, there exist $\epw, \epe, C>0$ such that the following is true. 
    
    \noindent
    Assume Assumption \ref{assumption}. Then for any $t\in[0,T]$,
    \begin{equation*}
        \frac{d}{dt} \int a(t,x) \, \eta\bigl(U(t,x)|\utl(t, x)\bigl) \,dx \leq \mathcal{Z}(U) - \mathcal{D}(U) + \mathcal{H}(U) + \mathcal{E}(U)   \, ,
    \end{equation*}
    where  
    \begin{equation*}
        \begin{aligned}
            \mathcal{Z}(U)&=\begin{cases}
                \dot{\x}_1\mathcal{Y}_{1}(U) + \dot{\x}_n\mathcal{Y}_{n}(U)\, , \textup{ if } W_n = S_n \, , \\
                \dot{\x}_1\mathcal{Y}_{1}(U) \, , \textup{ if } W_n = R_n \, ,
            \end{cases} \\
            \mathcal{D}(U) &= (1+2\gamma)\mathcal{D}_{x_1}^{p}(U) + C \mathcal{D}_{x_1}^{r}(U) + C \mathcal{D}_y(U)\, , \\
            \mathcal{H}(U) &= \bigl(-\mathbf{C}_1\textup{min}\{\Lambda_{S_1},\Lambda_{W_n}\} + \frac{\mathbf{C}_2}{\de}\bigl) \mathcal{H}_C \\
            &\quad + (1+\de) \mathcal{H}_{S_1} 
            \begin{cases}
                + (1+\de) \mathcal{H}_{S_n} \, , \textup{ if } W_n = S_n \, , \\
                - (1-\de) \mathcal{H}_{R_n} \, , \textup{ if } W_n = R_n \, ,
            \end{cases}\\
            \mathcal{E}(U) &= \begin{cases}
                 \frac{C}{\de} (\epws^2 \eps e^{-C\eps t} + \epws \eps^2 e^{-C\epws t}) \, , \textup{ if } W_n = S_n \, , \\
                 C \epe \epr\eps e^{-C\eps t} + C \epe \|\partial_{x_1} R_n\|_{L^4}^2 + \frac{C}{\gamma^{1/3}} \epe^{2/3}\|\partial_{x_1}^2 R_n\|_{L^{1}}^{4/3}  \, , \textup{ if } W_n = R_n \, ,
            \end{cases} 
        \end{aligned}
    \end{equation*}
    where $\gamma, \mathbf{C}_1, \mathbf{C}_2 >0$ are constants that depend only on $B_j,f,\eta,\ul$, and
    \begin{align*}
        \mathcal{Y}_{i}(U) &= \int \partial_{x_1} a_{S_i} ^ {\mathbf{X}_i} \, \eta({U|\Tilde{U}})\, dx - \int a \eta''(\Tilde{U})(U-\Tilde{U}) \partial_{x_1} S_i^{\mathbf{X}_i}\, dx \, ,\\
        \mathcal{D}_{x_1}^p(U) &= {B_1(\ul)\lo \cdot \lo}\int |\partial_{x_1}{\mu_1}| ^2 \, dx + {B_1(\ul){\lnn}\cdot{\lnn}}\int |\partial_{x_1} \nu_n | ^2 \, dx \, ,\\ 
        \mathcal{D}_{x_1}^{r}(U) &= \sum_{i=2}^n \int |\partial_{x_1} \mu_i | ^2 \, dx + \sum_{i=1}^{n-1} \int |\partial_{x_1} \nu_i | ^2 \, dx \, ,\\ 
        \mathcal{D}_y(U) &= \sum_{j=2}^d \int \big|\partial_{x_j}(U-\utl)\big|^2 \, dx \, , \\
        \mathcal{H}_C(U) &= \sum_{i=2}^{n}\eps\int \mu_i ^ 2 \dkoo dx + \sum_{i=1}^{n-1}\epww \int \nu_i ^2 \dkww \, dx \, ,  \\
        \mathcal{H}_{S_1}(U) &= -\cf \eps \int \mu_1 ^ 2 \dkoo \, dx \, , \\
        \mathcal{H}_{W_n}(U) &= -\cfs \epww \int \nu_n ^ 2 \dkww \, dx \, .
    \end{align*}
\end{lemm}

\noindent
\textit{Remark.} Recall (\ref{r curve opposite s curve}). We have
\begin{align*}
    \mathcal{D}_{x_1}^p \, , \; \mathcal{D}_{x_1}^r \, , \; \mathcal{D}_y \, , \; \mathcal{H}_C \, , \; \mathcal{H}_{S_1} \, , \; \mathcal{H}_{W_n} \geq 0 \, .
\end{align*}
We have four families of terms: the shift family $\mathcal{Z}$, the viscous family $\mathcal{D}$, the hyperbolic family $\mathcal{H}$, and the interaction family $\mathcal{E}$. 

The shift family $\mathcal{Z}$ corresponds to the new terms induced by the shifts.
The viscous family $\mathcal{D}$ comes from the dissipation. 
The viscous term $\mathcal{D}_{y}$ corresponds to the transverse direction, and the viscous terms $\mathcal{D}_{x_1}^p, \mathcal{D}_{x_1}^r$ correspond to the $x_1$ direction. If we examine $\mathcal{D}_{x_1}^p, \mathcal{D}_{x_1}^r$ in the phase space, $\mathcal{D}_{x_1}^p$ corresponds to the special direction $\ro$ (respectively $\rn$) of the wave $S_1$ (respectively $W_n$), and $D_{x_1}^r$ corresponds to the other orthogonal directions. In Lemma \ref{lemma parabolic} in the next section, we see how $D_{x_1}^p$ helps us get the Poincar\'{e} type inequality.

The flux functions give the hyperbolic family $\mathcal{H}$. The hyperbolic term $\mathcal{H}_{S_1}$ (respectively $\mathcal{H}_{W_n}$) corresponds to the special direction $\ro$ (respectively $\rn$) of the wave $S_1$ (respectively $W_n$). The hyperbolic term $\mathcal{H}_C$ corresponds to the other orthogonal directions. The hyperbolic ``scalarization" happens in all directions except the special direction $\ro$ (respectively $\rn$) of the wave $S_1$ (respectively $W_n$). More precisely, the coefficient of $\mathcal{H}_C$ becomes negative if we choose sufficiently large $\Lambda_{S_1}, \Lambda_{W_n}$, while the coefficients of $\mathcal{H}_{S_1},\mathcal{H}_{W_n}$ are independent of the choice of $\Lambda_{S_1}, \Lambda_{W_n}$. In short, the hyperbolic ``scalarization" reduces the problem to the scalar case with hyperbolic remainders in the form of $\mathcal{H}_{S_1}, \mathcal{H}_{W_n}$.

Such hyperbolic remainder is negative if it corresponds to a planar rarefaction wave, but it is positive if it corresponds to a planar viscous shock wave. In Lemma \ref{lemma parabolic} in the next section, we show in detail how the dissipation $D_{x_1}^p$, the shift family $\mathcal{Z}$, and the hyperbolic terms corresponding to the other orthogonal directions $\mathcal{H}_C$ control the positive hyperbolic remainder by the Poincar\'{e} type inequality. While the dissipation $D_{x_1}^p$ dominates the $L^2$ norm of the derivative, the shifts will be chosen in a way that the shift family $\mathcal{Z}$ controls the average with the help of the hyperbolic terms corresponding to the other orthogonal directions $\mathcal{H}_C$. 

Finally, we can bound interaction terms corresponding to $E_1,E_2$ in (\ref{eq for superposition wave}) in the form of $\mathcal{E}$ by Lemma \ref{interaction bound}.


\begin{proof}
By the definition of relative entropy function (\ref{eq for REF}), (\ref{eq for sol}), and (\ref{eq for superposition wave}), we have
\begin{align*}
     &\frac{d}{dt} \int a(t,x) \, \eta\bigl(U(t,x)|\Tilde{U}(t, x)\bigl) \,dx \\
     &= \int \partial_t a \, \eta({U|\Tilde{U}})\, dx + \int a \,\Bigr[\bigl(\eta'(U) - \eta'(\Tilde{U})\bigl)\partial_t U -  \eta'' (\Tilde{U})(U-\Tilde{U})\partial_t \Tilde{U}\Bigr ] dx \\ 
     &= \int \partial_t a \, \eta({U|\Tilde{U}})\, dx + \int a \, \biggr [\bigl(\eta'(U) - \eta'(\Tilde{U})\bigl)\Bigl( \sum_{j=1}^d\partial_{x_j}\bigl(B_j(U)\partial_{x_j} \eta'(U)\bigl) \\&\quad - \partial_x f(U) - \sum_{j=2}^d g_j(U)\Bigl) 
     - \eta''(\Tilde{U})(U-\Tilde{U})\Bigl(\partial_{x_1} \bigl(B_1(\Tilde{U})\partial_{x_1} \eta' (\Tilde{U})\bigl) \\ &\quad  -\partial_x f(\Tilde{U}) + Z + E_1 + E_2\Bigl)\biggr] \,dx \, .
\end{align*}
According to definitions of (\ref{eq for relative flux}) and (\ref{eq for flux of RE}), we have
\begin{equation*}
    \frac{d}{dt} \int a(t,x)\, \eta\bigl(U(t,x)|\Tilde{U}(t, x)\bigl) \, dx = \mathcal{Z}_1 + \sum_{i=1}^{4} \mathcal{D}_i + \sum_{i=1}^{2} \mathcal{E}_i + \sum_{i=1}^{4} \mathcal{H}_i\, ,
\end{equation*}
where the shift term is
\begin{equation*}
    \begin{aligned}
        \mathcal{Z}_1&= \int \partial_t a \, \eta({U|\utl})\,dx  - \begin{cases}
            \sum_{i=1,n}\dot{\mathbf{X}}_i\int a\, \eta''(\utl)(U-\utl) \partial_{x_1} S_i^{\mathbf{X}_i}\, dx  \, , \textup{ if } W_n =S_n \, , \\
            \dot{\mathbf{X}}_1\int a\, \eta''(\utl)(U-\utl) \partial_{x_1} S_1^{\mathbf{X}_1}\, dx \, , \textup{ if } W_n = R_n \, ,
        \end{cases}  
    \end{aligned}
\end{equation*}
the dissipation gives
\begin{align*}
    \mathcal{D}_1 &= \int a\, \bigl(\eta'(U)-\eta'(\Tilde{U})\bigl)\partial_{x_1}\Bigl(B_1(U)\partial_{x_1} \bigl(\eta'(U)-\eta'(\Tilde{U})\bigl)\Bigl) \,dx\, ,\\
    \mathcal{D}_2 &= \int a\, \bigl (\eta'(U)-\eta'(\Tilde{U})\bigl)\partial_{x_1}\Bigl(\bigl(B_1(U)-B_1(\Tilde{U})\bigl)\partial_{x_1} \eta'(\Tilde{U})\Bigl) \,dx\, , \\
    \mathcal{D}_3 &= \int a \, \eta' (U|\Tilde{U}) \partial_{x_1}\bigl(B_1(\Tilde{U})\partial_{x_1} \eta'(\Tilde{U})\bigl) \,dx\, , \\
    \mathcal{D}_4 &= \sum_{j=2}^d\int a \,\bigl(\eta'(U)-\eta'(\Tilde{U})\bigl)\partial_{x_j}\bigl(B_j(U)\partial_{x_j} \eta'(U)\bigl) \,dx\, , 
\end{align*}
the wave interaction causes
\begin{align*}
    \mathcal{E}_1 & = -\int a \, \eta''(\Tilde{U})(U-\Tilde{U}) E_1 \,dx\, , \\
    \mathcal{E}_2 & = -\int a \, \eta''(\Tilde{U})(U-\Tilde{U}) E_2 \,dx\, ,
\end{align*}
and the flux induces
\begin{align*}
    \mathcal{H}_1 &= \int \partial_{x_1} a \, q_1(U;\Tilde{U})\,dx \, , \\
    \mathcal{H}_2 &= -\int a \, \partial_{x_1} \eta'(\Tilde{U}) f(U|\Tilde{U}) \,dx \, , \\
    \mathcal{H}_3 &= \sum_{j=2}^d\int \partial_{x_j} a \, q_j(U;\Tilde{U})\,dx \, , \\
    \mathcal{H}_4 &= -\sum_{j=2}^d\int a \, \partial_{x_j} \eta'(\Tilde{U}) g_{j}(U|\Tilde{U}) \,dx \, . 
\end{align*}
We are going to analyze those terms and get $\mathcal{Z}, \mathcal{D}, \mathcal{H},\mathcal{E}$ defined in the lemma. 

\noindent
\textbf{Hyperbolic parts and dissipation in $y$.} As $\so$ and $\wn$ are planar waves, 
$a$ and $\utl$ do not depend on $y$. Also, we know $B_j$ are positive definite and $\eta''$ is strictly convex. Therefore, we have
\begin{align*}
     \mathcal{H}_3 &= \mathcal{H}_4 = 0 \, , \\
     \mathcal{D}_4 &= - \sum_{j=2}^d \int a \, \partial_{x_j}\bigl(\eta'(U) - \eta'(\utl)\bigl)B_j(U)\partial_{x_j} \bigl(\eta'({U})- \eta'(\utl)\bigl)\,dx \\ 
     & \leq - C D_{y}\, .
\end{align*}

\noindent
\textbf{Dissipation in $x_1$.}
By integration by parts, we have
\begin{align*}
    \mathcal{D}_1
    =& - \int a \, \partial_{x_1}\bigl(\eta'(U)-\eta'(\Tilde{U})\bigl) B_1(U) \partial_{x_1} \bigl(\eta'(U)-\eta'(\Tilde{U})\bigl) \,dx\\
    & - \int \partial_{x_1} a \,  \bigl(\eta'(U)-\eta'(\Tilde{U})\bigl)B_1(U)\partial_{x_1} \bigl(\eta'(U)-\eta'(\Tilde{U})\bigl) \,dx\\
    =&: \,\mathcal{D}_{11} + \mathcal{D}_{12}\, , \\
    \mathcal{D}_2
    =& - \int a \, \partial_{x_1}\bigl(\eta'(U)-\eta'(\Tilde{U})\bigl)\bigl(B_1(U)-B_1(\Tilde{U})\bigl)\partial_{x_1} \eta'(\Tilde{U}) \,dx \\
    & - \int \partial_{x_1} a \, \bigl (\eta'(U)-\eta'(\Tilde{U})\bigl)\bigl(B_1(U)-B_1(\Tilde{U})\bigl)\partial_{x_1} \eta'(\Tilde{U}) \,dx \\
    =&: \,\mathcal{D}_{21} + \mathcal{D}_{22}\, .
\end{align*}
We apply projections (\ref{decomposition of S_1}) and (\ref{decomposition of W}) to $\mathcal{D}_{11}$. By (\ref{L infinity}) and the fact that $B_1$ is positive definite, we take $\epe,\epw$ small enough and get
\begin{align*}
    \mathcal{D}_{11} &\leq -\bigl(1-C(\epe+\epw)-C (\Lambda_{S_1}+\Lambda_{W_n})\epw\bigl) \\ 
    & \quad \Bigl(\int \bigl(\lo \partial_{x_1}\mu_1 + \sum_{i=2}^n \voi \partial_{x_1}\mu_i\bigl)B_1(\ul) \bigl(\lo \partial_{x_1}\mu_1 +\sum_{i=2}^n  \voi \partial_{x_1}\mu_i\bigl)\, dx \\
    &\quad  + \int \bigl(\lnn \partial_{x_1}\nu_n + \sum_{i=1}^{n-1}  \vti \partial_{x_1}\nu_i\bigl)B_1(\ul) \bigl(\lnn \partial_{x_1}\nu_n + \sum_{i=1}^{n-1} \vti \partial_{x_1}\nu_i\bigl)\, dx\Bigl)
    \\&\leq - \bigl(1-C(\epe+\epw^{1/2})\bigl)(\mathcal{D}_{x_1}^{p} + \mathcal{D}_{x_1}^{r})\, .
\end{align*}
Compared with $\mathcal{D}_{11}$, $\partial_{x_1} a$ and $\partial_{x_1} \eta'(\Tilde{U})$ give extra smallness to $\mathcal{D}_{12}, \mathcal{D}_{21}, \mathcal{D}_{22}$.By Young's inequality, (\ref{k_S as main part}), and (\ref{k_R as main part}), we take $\epe,\epw$ small enough and get
\begin{align*}
    \mathcal{D}_1 + \mathcal{D}_2 
    &\leq  - \bigl(1-C(\epe+\epw^{1/4})\bigl)(\mathcal{D}_{x_1}^{p} + \mathcal{D}_{x_1}^{r})  \\&\quad + C \epw^{1/2}\int \bigl(|\partial_{x_1} \so| + |\partial_{x_1} W_n^{\x_n}|\bigl) |U-\Tilde{U}|^2 \, dx\, .
\end{align*}
We have
\begin{align*}
    \eta'(U)-\eta'(\utl) = \mu_1 \lo + \sum_{i=2}^n\mu_i  \voi = \nu_n \lnn + \sum_{i=1}^{n-1}\nu_i \vti \, ,
\end{align*}
which implies
\begin{align*}
    &\mu_1 \lo\cdot \ro + \sum_{i=2}^n\mu_i  \voi\cdot \ro = \sum_{i=1}^{n-1}\nu_i \vti \cdot \ro \, , \\
    &\sum_{i=2}^n\mu_i  \voi\cdot \rn = \nu_n \lnn \cdot \rn+ \sum_{i=1}^{n-1}\nu_i \vti \cdot \rn\, .
\end{align*}
As $\lo\cdot \ro, \lnn\cdot \rn \neq 0$, we get
\begin{align*}
    \sum_{i=2}^n|\partial_{x_1}\mu_i|^2 + \sum_{i=1}^{n-1}|\partial_{x_1}\nu_i|^2 \geq C \bigl( |\partial_{x_1}\mu_1|^2 + |\partial_{x_1}\nu_n|^2 \bigl)\, ,
\end{align*}
for some $C>0$. Thus there exists some constant $\gamma>0$ that depends only on $B_j,f,\eta,\ul$ such that
\begin{align*}
    \frac{1}{2}\mathcal{D}_{x_1}^{r} \geq 4\gamma \mathcal{D}_{x_1}^{p} \, .
\end{align*}
Taking $\epe, \epw$ small enough, we get
\begin{align*}
    \mathcal{D}_1 + \mathcal{D}_2 
    &\leq - (1+3\gamma)\mathcal{D}_{x_1}^{p} - \frac{1}{4} \mathcal{D}_{x_1}^{r}  \\ &\quad + C \epw^{1/2} \int \bigl(|\partial_{x_1} \so| + |\partial_{x_1} W_n^{\x_n}|\bigl) |U-\Tilde{U}|^2 \, dx\, .
\end{align*}
By the chain rule, Lemma \ref{pointwise bound of shock}, Lemma \ref{property of rarefaction}, and (\ref{L infinity}), we have
\begin{align*}
    |\mathcal{D}_3| \leq  C \epw \int \bigl(|\partial_{x_1} \so| + |\partial_{x_1} W_n^{\x_n}|\bigl) |U-\Tilde{U}|^2 \,dx + C \epe \int |U-\utl||\partial_{x_1}^2 R_n| \,dx\, .
\end{align*}

\noindent
\textbf{Hyperbolic parts in $x_1$.}
By the definition of the flux of the relative entropy (\ref{eq for flux of RE}), (\ref{L infinity}), (\ref{k_S as main part}), and (\ref{k_R as main part}), we take $\epe, \epw$ small enough and get
\begin{align*}
    \mathcal{H}_1 &= \int \partial_{x_1} a \, q_1(U;\Tilde{U})\,dx\\
    &= \frac{1}{2} (I_1+I_2) + C(\epe^{1/2}+\epw^{1/2}) \int \bigl(|\partial_{x_1} \so| + |\partial_{x_1} W_n^{\x_n}|\bigl) |U-\Tilde{U}|^2 \, dx \, ,
\end{align*}
where
\begin{align*}
    I_1=\int \partial_{x_1} \asxo (U-\Tilde{U}) \eta''(\ul)f'(\ul)(U-\Tilde{U}) \, dx \, , \\
    I_2=\int \partial_{x_1} \aw (U-\Tilde{U}) \eta''(\ul)f'(\ul)(U-\Tilde{U}) \, dx \, .
\end{align*}
We apply (\ref{decomposition of S_1 taylor}) to $I_1$. Recall (\ref{def of projection}). Let 
\begin{align*}
    v = \sum_{i=2}^n \mu_i \eta''(\ul)^{-1} \voi \in V \, .
\end{align*}
By (\ref{trick eta'' f'}), (\ref{decomposition of S_1 taylor}), (\ref{property of perturbation 1}), and (\ref{property of perturbation 2}), we take $\epe, \epw$ small enough and get
\begin{align*}
    I_1 &= \int \partial_{x_1} \asxo (U-\Tilde{U}) \eta''(\ul)\bigl(f'(\ul)-\lambda_1 I\bigl)(U-\Tilde{U}) \, dx + 2 \mathcal{Z}_2\\
    &\leq - {\Lambda_{S_1}} (\lambda_2-\lambda_1) \eps \int \eta'' (\ul)P(v) \cdot P(v) \dkoo \, dx \\
    &\quad + C (\epe^{1/2}+\epw^{1/2}) \int |\partial_{x_1} \so| |U-\Tilde{U}|^2 \, dx + 2 \mathcal{Z}_2\\
    &\leq - C {\Lambda_{S_1}} (\lambda_2-\lambda_1)\sum_{i=2}^n \eps  \int \mu_i^2 \dkoo \, dx  \\
    &\quad + C (\epe^{1/2}+\epw^{1/2}) \int |\partial_{x_1} \so| |U-\Tilde{U}|^2 \, dx + 2 \mathcal{Z}_2\, ,
\end{align*}
where
\begin{align*}
    \mathcal{Z}_2 = \frac{\lambda_1}{2} \int \partial_{x_1} \asxo (U-\Tilde{U}) \eta''(\ul)(U-\Tilde{U}) \, dx \, .
\end{align*}
Similarly, we get
\begin{align*}
    I_2& = \int \partial_{x_1} \aw (U-\Tilde{U}) \eta''(\ul)\bigl(f'(\ul)-\lambda_n I\bigl)(U-\Tilde{U}) \, dx + 2\mathcal{Z}_3 \\
    &\leq C \Lambda_{W_n} (\lambda_{n-1}-\lambda_n) \sum_{i=1}^{n-1} \epww \int \nu_i ^ 2 \dkww \, dx  \\
    &\quad + C (\epe^{1/2}+\epw^{1/2}) \int |\partial_{x_1} \wn| |U-\Tilde{U}|^2 \, dx + 2\mathcal{Z}_3 \, , 
\end{align*}
where
\begin{align*}
    \mathcal{Z}_3 = \frac{\lambda_n}{2}\int \partial_{x_1} \aw (U-\Tilde{U}) \eta''(\ul)(U-\Tilde{U}) \, dx \, .
\end{align*}
As $\lambda_2-\lambda_1, \lambda_n-\lambda_{n-1}>0$, we get
\begin{align*}
    \mathcal{H}_1&\leq-C \textup{min}\{\Lambda_{S_1},\Lambda_{W_n}\} \mathcal{H}_{C}+ \mathcal{Z}_2 + \mathcal{Z}_3\\ &\quad + C(\epe^{1/2}+\epw^{1/2}) \int \bigl(|\partial_{x_1} \so|+|\partial_{x_1} \wn|\bigl)|U-\Tilde{U}|^2 \, dx  \, .
\end{align*}
By the definition of the relative flux (\ref{eq for relative flux}) and (\ref{L infinity}), we take $\epe, \epw$ small enough and get
\begin{align*}
    \mathcal{H}_2 &= -\int a \, \partial_{x_1} \eta'(\Tilde{U}) f(U|\Tilde{U}) \,dx  \\
    &\leq I_3+I_4 + C(\epe+\epw^{1/2})\int \bigl(|\partial_{x_1} \so|+|\partial_{x_1} \wn|\bigl) |U-\utl|^2\,dx\, ,
\end{align*}
where
\begin{align*}
    I_3 &= - \eps \int \dkoo \,\lo\cdot f''(\ul) : (U-\utl)\otimes (U-\utl) \,dx \, , \\
    I_4 &= 
    \begin{cases}
        -\epww \int  \dktt \,\lnn \cdot f''(\ul) : (U-\utl)\otimes (U-\utl) \,dx \, , \textup{ if } W_n=S_n \, , \\
        \epww \int  \dkrr \, \lnn \cdot f''(\ul) : (U-\utl)\otimes (U-\utl) \,dx \, , \textup{ if } W_n=R_n \, .
    \end{cases}
\end{align*}
We apply (\ref{decomposition of S_1 taylor}) to $I_3$. As the hyperbolic ``scalarization" will happen in all the directions except the special direction $\ro$ of the wave $\so$, we use Young's inequality to make the hyperbolic term in the $\ro$ direction as small as possible
\begin{equation*}
    \begin{aligned}
        & I_3 -C(\epe+\epw)\int \bigl(|\partial_{x_1} \so|+|\partial_{x_1} \wn|\bigl) |U-\utl|^2\,dx 
        \\ &\leq (1+\frac{\de}{2}) \mathcal{H}_{S_1} + \frac{C}{\de} \sum_{i=2}^{n} \eps \int \mu_i^2 \dkoo \, dx  \, .
    \end{aligned}
\end{equation*}
Similarly, we get
\begin{align*}
    &I_4 - C(\epe+\epw)\int \bigl(|\partial_{x_1} \so|+|\partial_{x_1} \wn|\bigl) |U-\utl|^2\,dx
    \\ &\leq \frac{C}{\de} \sum_{i=1}^{n-1} \epww \int \nu_i^2 \dkww \, dx  
        \begin{cases}
            + (1+\frac{\de}{2}) \mathcal{H}_{S_n} \, , \textup{ if } W_n = S_n \, ,\\
            -(1-\frac{\de}{2}) \mathcal{H}_{R_n} \, , \textup{ if } W_n = R_n \, . 
        \end{cases}
\end{align*}
Then we have
\begin{equation*}
    \begin{aligned}
        &\mathcal{H}_2 - C(\epe+\epw^{1/2})\int \bigl(|\partial_{x_1} \so|+|\partial_{x_1} \wn|\bigl)|U-\utl|^2\,dx \\ &\leq \frac{C}{\de}\mathcal{H}_{C} + (1+\frac{\de}{2}) \mathcal{H}_{S_1} \begin{cases}
            + (1+\frac{\de}{2}) \mathcal{H}_{S_n} \, , \textup{ if } W_n = S_n \, ,\\
            -(1-\frac{\de}{2}) \mathcal{H}_{R_n} \, , \textup{ if } W_n = R_n \, .
        \end{cases}
    \end{aligned}
\end{equation*}
In all, we get
\begin{equation*}
    \begin{aligned}
        &\mathcal{H}_1 + \mathcal{H}_2 - C (\epe^{1/2}+\epw^{1/2}) \int \bigl(|\partial_{x_1} \so|+|\partial_{x_1} \wn|\bigl)|U-\Tilde{U}|^2 \, dx - \mathcal{Z}_2 -\mathcal{Z}_3\\ 
        &\leq \bigl(-C\textup{min}\{\Lambda_{S_1},\Lambda_{W_n}\} + \frac{C}{\de}\bigl) \mathcal{H}_C + (1+\frac{\de}{2}) \mathcal{H}_{S_1} \\
        &\quad 
        \begin{cases}
            +(1+\frac{\de}{2}) \mathcal{H}_{S_n} \, , \textup{ if } W_n = S_n \, ,\\
            -(1-\frac{\de}{2}) \mathcal{H}_{R_n} \, , \textup{ if } W_n = R_n \, .
        \end{cases}
    \end{aligned}
\end{equation*}

\noindent
\textbf{Shift terms}. We have
\begin{equation*}
    \begin{aligned}
        \partial_t a = (\dot{\x}_1-\sio) \partial_{x_1} a_{S_1}^{\x_1} +
        \begin{cases}
            (\dot{\x}_n - \sit) \partial_{x_1} a_{S_n}^{\x_n} \, , \textup{ if } W_n = S_n \, , \\
            \Lambda_{R_n}\partial_{x_1} f(R_n) \cdot \lnn \, , \textup{ if } W_n = R_n \, .
        \end{cases}
    \end{aligned}
\end{equation*}
By (\ref{k_S as main part}) and (\ref{k_R as main part}), we take $\epe, \epw$ small enough and get
\begin{align*}
    \mathcal{Z}_1 + \mathcal{Z}_2 + \mathcal{Z}_3 &= \mathcal{Z} + \mathcal{Z}_4 + C(\epe^{1/2}+\epw^{1/2})\int \bigl(|\partial_{x_1} \so|+|\partial_{x_1} \wn|\bigl)|U-\utl|^2\,dx \, ,
\end{align*}
where
\begin{equation*}
    \begin{aligned}
        \mathcal{Z}_4 = \begin{cases}
            0 \, , \textup{ if } W_n = S_n \, ,\\ 
            \Lambda_{R_n} \int \bigl(\partial_{x_1} f(R_n) -\lambda_n \partial_{x_1} R_n\bigl)\cdot \lnn \, \eta(U|\Tilde{U}) \, dx\, , \textup{ if } W_n = R_n \, .
        \end{cases} \, .
    \end{aligned}
\end{equation*}
By Taylor expansion and Lemma \ref{property of rarefaction}, we take $\epe, \epw$ small enough and get
\begin{align*}
    |\mathcal{Z}_4| &\leq C \Lambda_{R_n} \int \big|\bigl(f'(\ul)-{\lambda_n}I\bigl)\partial_{x_1} R_n \cdot \lnn\big| |U-\utl|^2 \, dx \\
    &\quad + C \Lambda_{R_n} \epw \int |\partial_{x_1}R_n| |U-\utl|^2 \, dx \\
    &\leq C \epw^{1/2} \int |\partial_{x_1} R_n| |U-\utl|^2 \, dx \, .
\end{align*}
In all, we get
\begin{align*}
    \mathcal{Z}_1 + \mathcal{Z}_2 + \mathcal{Z}_3\leq \mathcal{Z} + C(\epe^{1/2}+\epw^{1/2})\int \bigl(|\partial_{x_1} \so|+|\partial_{x_1} \wn|\bigl)|U-\utl|^2 \,dx\, .
\end{align*}

\noindent
\textbf{Interaction terms.} Note $\mathcal{H}_C, \mathcal{H}_{S_1}, \mathcal{H}_{W_n}\geq 0$. We have
\begin{align}\label{commen error}
    \int \bigl(|\partial_{x_1} \so|+|\partial_{x_1} \wn|\bigl)|U-\utl|^2 \leq C (\mathcal{H}_C + \mathcal{H}_{S_1} + \mathcal{H}_{W_n}) \, .
\end{align}
If $W_n=S_n$, then Lemma \ref{lemm reorganization} and Lemma \ref{interaction bound} give
\begin{align*}
    &|\mathcal{E}_1| + |\mathcal{E}_2| \\
    &\leq C \int |U-\utl| \Bigl(|\partial_{x_1} \so||\st-\um| + |\partial_{x_1} \st||\so-\um| \\&\quad +|\partial_{x_1} \st||\partial_{x_1} \so |\Bigl) \, dx\\
    &\leq C \Bigl(\int \bigl(|\partial_{x_1}\so| + |\partial_{x_1}\st|\bigl) |U-\utl|^2 \, dx\Bigl)^{1/2} \Bigl(\big\||\partial_{x_1} \so| |\st-\um|^2\big\|_{L^1}^{1/2} 
    \\&\quad + \big\||\partial_{x_1} \st| |\so-\um|^2 \big\|_{L^1}^{1/2} + \big\||\partial_{x_1} \st||\partial_{x_1} \so |\big\|_{L^1}^{1/2}\Bigl) \\
    &\leq \frac{\de}{4}(\mathcal{H}_C + \mathcal{H}_{S_1} + \mathcal{H}_{W_n})  +\frac{C}{\de} (\epws^2 \eps e^{-C\eps t} + \epws \eps^2 e^{-C\epws t})
    \, .
\end{align*}
If $W_n=R_n$, then (\ref{L infinity bound}), Lemma \ref{lemm reorganization}, and Lemma \ref{interaction bound} give
\begin{align*}
    &|\mathcal{E}_1| + |\mathcal{E}_2| \\
    & \leq C \int |U-\Tilde{U}| \Bigl(|\partial_{x_1}^2 R_n| + |\partial_{x_1} R_n|^2 + |\partial_{x_1} R_n||\partial_{x_1} \so|\\
    &\quad + |\partial_{x_1} R_n||\so-\um| + |\partial_{x_1} \so||R_n-\um|\Bigl) \, dx \\
    &\leq C \int |U-\Tilde{U}| |\partial_{x_1}^2 R_n| + C\epe \Bigl(\|\partial_{x_1} R_n\|_{L^4}^2+ \big\||\partial_{x_1} R_n||\partial_{x_1} \so|\big\|_{L^2} \\ 
    &\quad + \big\||R_n-\um| |\partial_{x_1} \so|\big\|_{L^2} + \big\||\so-\um| |\partial_{x_1} R_n|\big\|_{L^2} \Bigl)
    \\
    &\leq C_E \int |U-\Tilde{U}| |\partial_{x_1}^2 R_n| + C \epe \|\partial_{x_1} R_n\|_{L^4}^2 + C\epe \epr\eps e^{-C\eps t}\,  ,
\end{align*}
where $C_E>0$ is a constant that depends only on $B_j,f,\eta,\ul$.

\noindent
\textbf{Some estimates.}
By H\"{o}lder's inequality and the Gagliardo-Nirenberg inequality proved in \cite{Ninequality}, we have 
\begin{equation*}
    \int |U-\utl||\partial_{x_1}^2 R_n| \, dx \leq
    C \big\|\partial_{x_1}(U-\utl)\big\|_{L^2}^{1/2} \|U-\utl\|_{L^2}^{1/2} 
    \|\partial_{x_1}^2 R_n\|_{L^{1}} \, . 
\end{equation*}
The Young's inequality and (\ref{L infinity bound}) give
\begin{equation*}
    2C_E\int |U-\utl||\partial_{x_1}^2 R_n| \, dx \leq
    \frac{C}{\gamma^{1/3}}\epe^{2/3} \|\partial_{x_1}^2 R_n\|_{L^{1}}^{4/3} + {\gamma} D_{x_1}^{p} + \frac{1}{8} D_{x_1}^r \, .
\end{equation*}
Finally, we can take $\epe,\epw$ small enough and prove the lemma because of (\ref{commen error}).
\end{proof}
\subsection{Poincar\'{e} type inequality}\label{subsec poincare} 
\begin{lemm}\label{lemma parabolic}
    For any $\ul \in \mathbb{R}^n$, any $0<\gamma<1$, and any $\Lambda_{S_1},\Lambda_{W_n}, \Tilde{C}_1, \Tilde{C}_n>0$, there exist $\epw, \epe, C>0$ such that the following is true. 
    
    \noindent 
    Assume Assumption \ref{assumption}. Let $\mn_i$ be $\mu_1$ if $i=1$ and be $\nu_n$ if $i=n$. Then for any $i\in\{1,n\}$ and any $t\in[0,T]$,
    \begin{align*}
    &\mathcal{H}_{S_i}(U) - (1+\gamma) D_{x_1}^{p,i}(U) + \dot{\x}_i(t) \mathcal{Y}_i(U) \\ 
    &\leq -C \gamma \int |\partial_{x_1}{\mn_i}| ^2 \, dx- C \gamma \epss \int \mn_i ^ 2 \dkk \, dx- \frac{\epss}{2\Tilde{C}_i}|\dot{\mathbf{X}}_i(t)| ^2  \\
    &\quad + \bigl(- \frac{\mathbf{C}_3}{\epss} \Tilde{C}_i  + \frac{\mathbf{C}_4}{\epss} \bigl)\bigl(\epss \int \mn_i \dkk \, dx \bigl)^2
    + \mathbf{C}_5 \Tilde{C}_i \mathcal{H}_C(U) + C \epss D_y(U)\, ,
    \end{align*}
    where $\mathbf{C}_3, \mathbf{C}_4, \mathbf{C}_5>0$ are constants that depend only on $B_j,f,\eta,\ul$, 
    \begin{align*}
    D_{x_1}^{p,i}(U) &= {B_1(\ul)\li \cdot \li} \int |\partial_{x_1}{\mn_i}| ^2 \, dx \, ,
    \end{align*}
    and $\mathcal{H}_{S_i},\mathcal{H}_C,\mathcal{D}_y, \mathcal{Y}_i$ are defined in Lemma \ref{lemma rem}.
\end{lemm}
\noindent
\textit{Remark}. By taking $\Tilde{C}_i$ large enough, the positive remainders are $\mathbf{C}_5 \Tilde{C}_i \mathcal{H}_C$ and $C \epss D_y(U)$. We will take $\Lambda_{S_1},\Lambda_{W_n}$ large enough to control $\mathbf{C}_5 \Tilde{C}_i \mathcal{H}_C$ and $\epw$ small enough to control $C \epss D_y(U)$ in Lemma \ref{lemma L^2 estimate} in subsection \ref{subsec L^2 contraction}.
\begin{proof}
First, we fix the time $t$ and the transverse direction $y$ and compactify the problem by changing variables. Let
\begin{align*}
    (\mathcal{H}_{S_i})^{t,y}  &= - \cfi \epss \int \mn_i\bigl(t,(x_1,y)\bigl) ^ 2 \dkk \, d{x_1} \, , \\
    (D_{x_1}^{p,i})^{t,y} &= {B_1(\ul)\li \cdot \li} \int \big|\partial_{x_1}{\mn_i}\bigl(t,(x_1,y)\bigl)\big| ^2 \, d x_1 \, .
\end{align*}
As ${k_{S_i}^{\x_i}}$ is strictly increasing, we define the change of variable
\begin{equation*}
    x_1\mapsto z:= {k_{S_i}^{\x_i}}\bigl(t,(x_1,y)\bigl) \text{ and } h(z) = \mn_i\bigl(t,(x_1,y)\bigl) \, .
\end{equation*}
By (\ref{derivative of k}), we can take $\epw$ small enough and get
\begin{align*}
    (\mathcal{H}_{S_i})^{t,y} &= - \cfi \epss \int_0^1 h(z) ^ 2  \, dz  \, , \\
    -(1+\gamma)(D_{x_1}^{p,i})^{t,y} &\leq (\frac{1}{2}+\frac{\gamma}{4}) \, \cfi \epss \int_0^1 h'(z)^2 z (1-z) \,dz \, .
\end{align*}
Lemma \ref{Poincare inequality lemma} gives
\begin{align*}
    \int_0^1 h(z)^2\, dz &= \int_0^1 \bigl(h(z) - \int_0^1 h(z)\,dz\bigl)^2  \, dz +  \bigl(\int_0^1 h(z)\,dz\bigl) ^ 2\\
    &\leq \frac{1}{2} \int_0^1 h'(z)^2 z(1-z)\, dz + \bigl(\int_0^1 h(z)\,dz\bigl) ^ 2\, ,
\end{align*}
which implies
\begin{align*}
    &(\mathcal{H}_{S_i})^{t,y}-(1+\gamma)(D_{x_1}^{p,i})^{t,y} \\ &\leq - C\gamma \epss \int \mn_i ^ 2 \dkk \, d x_1 -C\gamma \int |\partial_{x_1}{\mn_i}| ^2 \, d x_1  + \frac{C}{\epss} \bigl(\epss \int \mn_i \dkk \, d x_1\bigl)^2 \, ,
\end{align*}
where $C>0$ is independent of $t,y$.
Therefore, we have
\begin{align*}
    &\mathcal{H}_{S_i}(U) - (1+\gamma) D_{x_1}^{p,i}(U)\\ &\leq - C\gamma \epss \int \mn_i ^ 2 \dkk \, dx -C\gamma \int |{\partial_{x_1} \mn_i}| ^2 \, dx  + \frac{C}{\epss} \int \bigl(\epss \int \mn_i \dkk \, d{x_1}\bigl)^2 dy \, .
\end{align*}
Using the classical Poincar\'{e} inequality, we find
\begin{align*}
    &\int \bigl(\epss \int \mn_i \dkk \, d x_1\bigl)^2 \,dy \\
    &\leq  C\bigl(\epss \int \mn_i \dkk \, dx\bigl)^2 
    +C \int \big|\nabla_y\bigl(\epss \int \mn_i \dkk \, d x_1\bigl)\big|^2 \,dy \\
    &\leq  C\bigl(\epss \int \mn_i \dkk \, dx\bigl)^2 
    +C{\epss^2}\mathcal{D}_y\, .
\end{align*}
By Young's inequality, H\"{o}lder's inequality, (\ref{k_S as main part}), and (\ref{L infinity}), we take $\epe,\epw$ small enough and get
\begin{align*}
    (\frac{\epss}{\Tilde{C}_i})^2|\dot{\mathbf{X}}_i(t)| ^2 &
    = \bigl(\int a \, \eta''(\Tilde{U})(U-\Tilde{U}) \partial_{x_1} \si \, dx\bigl)^2 \\
    &\geq \bigl(\int  \bigl(\eta'(U) - \eta'(\Tilde{U})\bigl)  \partial_{x_1} \si  \, dx\bigl)^2 \\&\quad - C(\epe^{1/2}+\epw^{1/2})
    \bigl(\int |U-\Tilde{U}| |\partial_{x_1} \si| \, dx\bigl)^2 \\
    &\geq C \bigl(\epss \int \mn_i \dkk \, dx\bigl)^2 -  C\sum_{j\neq i} \bigl(\epss \int \mn_j \dkk \, dx\bigl)^2 \\ &\quad-C (\epe^{1/2}+\epw^{1/2}) \,\epss^2 \int |U-\Tilde{U}|^2 \dkk\, dx \\
    &\geq  C \bigl(\epss \int \mn_i \dkk \, dx \bigl)^2 - C \sum_{j\neq i}\epss ^2 \int \mn_j ^2 \dkk \, dx \\
    &\quad- C (\epe^{1/2}+\epw^{1/2}) \, \epss^2 \int \mn_i ^2 \dkk \, dx \, ,
\end{align*}
where $\mn_j=\mu_j$ if $i=1$ and $\mn_j=\nu_j$ if $i=n$ for any $j\neq i$.

By (\ref{k_S as main part}), (\ref{L infinity}), and H\"{o}lder's inequality, we take $\epe, \epw$ small enough and get
\begin{equation*}
    \begin{aligned}
        &\dot{\x}_i\mathcal{Y}_i + \frac{\epss}{2\Tilde{C}_i}|\dot{\mathbf{X}}_i(t)| ^2 \\
        &= -\frac{\epss}{2\Tilde{C}_i}|\dot{\mathbf{X}}_i(t)| ^2 + \dot{\mathbf{X}}_i(t) \int \partial_{x_1} a_{S_i}^{\x_i} \, \eta({U|\Tilde{U}})dx\\
        &\leq - \frac{\epss}{2\Tilde{C}_i}|\dot{\mathbf{X}}_i(t)| ^2 + C\epe \Lambda_{S_i} \frac{\Tilde{C}_i}{\epss} \bigl(\int |U-\Tilde{U}| |\partial_{x_1} \si| \, dx\bigl)^2 \\
        &\leq - \frac{C}{\epss} \Tilde{C}_i \bigl(\epss \int \mn_i \dkk \, dx \bigl)^2 + C \Tilde{C}_i \sum_{j\neq i} \epss \int \mn_j ^2 \dkk \, dx  \\&\quad + C (\epe^{1/4}+\epw^{1/4}) \, \epss \int \mn_i ^2 \dkk \, dx \, .
    \end{aligned}
\end{equation*}
We take $\epe,\epw$ small enough and get the estimate.
\end{proof}
\subsection{$L^2$ estimates}\label{subsec L^2 contraction}
Now it is time to choose $\Lambda_{S_1},\Lambda_{W_n},\Tilde{C}_1, \Tilde{C}_n$ that work for both Lemma \ref{lemma rem} and Lemma \ref{lemma parabolic}.
\begin{lemm} \label{lemma L^2 estimate}
    For any $\ul \in \mathbb{R}^n$, there exist $\epw, \epe, C, \Lambda_{S_1}, \Lambda_{W_n}, \Tilde{C}_1, \Tilde{C}_n>0$ such that the following is true. 

    \noindent
    Assume Assumption \ref{assumption}. Then for any $t \in [0,T]$,
     \begin{align*}
        \begin{split}
            \|U(t,\cdot)-\utl(t,\cdot)\|_{L^2(\mathbb{R}\times {\mathbb{T}^{d-1}})} ^ 2
            + \int_0^t D_0(U) + G_{S_1}(U) + G_{W_n}(U) + Y \, ds 
            \\ \leq  C \|U_0-\utl_0\|_{L^2(\mathbb{R}\times {\mathbb{T}^{d-1}})} ^2 + C E^2 \, ,
        \end{split}
     \end{align*}
    where $D_0, G_{S_1}, G_{W_n}, Y, E$ are defined in \textup{(\ref{def of D, D1, G_S, G_R})}.
\end{lemm}
\begin{proof}
We fix $\gamma$ as in Lemma \ref{lemma rem}. We choose $\de>0$ small enough such that
\begin{align*}
    \de<1 \, \text{ and } \,(1+\de)(1+\gamma)\leq 1+\frac{3}{2}\gamma \, .
\end{align*}
We take
\begin{align*}
    \Tilde{C}_i = \frac{2\mathbf{C}_4}{\mathbf{C}_3} \, .
\end{align*}
We take  $\Lambda_{S_1}, \Lambda_{W_n} = \Lambda$ where
\begin{align*}
    -\mathbf{C}_1 \Lambda + \frac{\mathbf{C}_2}{\de} + (1+\de)\mathbf{C}_5\Tilde{C}_i =  -1 \, .
\end{align*}
Finally, we take $\epe, \epw$ small enough such that Lemma \ref{lemma rem}, Lemma \ref{lemma parabolic}, and (\ref{rarefaction time bound}) give the estimate.
\end{proof}

\section{Proof of Prop \ref{prop energy estimate}}\label{H^m contraction}
\subsection{Some estimates}
Before proving the $H^m$ estimates, we bound the higher-order terms evaluated in Lemma \ref{lemm reorganization}.

Recall $m$ defined in (\ref{value of m}). Let $\phi: \mathbb{R}\times {\mathbb{T}^{d-1}} \rightarrow \mathbb{R}^n$ be a function. The Gagliardo-Nirenberg inequality for functions defined on $\mathbb{R}^d$ is proved in \cite{Ninequality}. Together with the periodicity in the transverse direction, we get the following two inequalities. If $1\leq q \leq m-1$, then for any $2\leq p < \frac{2d}{d-2q}$, there exists $C>0$ such that 
\begin{align}\label{gn}
    \|\phi\|_{L^p(\mathbb{R}\times {\mathbb{T}^{d-1}})} \leq C \|\phi\|_{H^q(\mathbb{R}\times {\mathbb{T}^{d-1}})} \, .
\end{align}
If $q = m$, then for any $2\leq p< \infty$, there exists $C>0$ such that 
\begin{align}\label{gn extreme}
    \|\phi\|_{L^p(\mathbb{R}\times {\mathbb{T}^{d-1}})} \leq C \|\phi\|_{H^q(\mathbb{R}\times {\mathbb{T}^{d-1}})} \, .
\end{align}
\begin{lemm}\label{lemm estimate in H^m contraction}
    Assume Assumption \ref{assumption}. Let $1\leq k\leq m$ where $m$ is defined in (\ref{value of m}).
    Then there exists a constant $C>0$ such that for any $t \in [0,T]$ and any $L\in \mathcal{L}^k$ where $\mathcal{L}^k$ is defined in \textup{(\ref{def of L^k})}, 
    \begin{align*}
        \|L\|_{L^2(\mathbb{R}\times {\mathbb{T}^{d-1}} )}^2 \leq C \epe D_k + C\sum_{j=0}^{k-1} D_j \, ,
    \end{align*}
    where $D_j$ are defined in \textup{(\ref{def of D, D1, G_S, G_R})}.
\end{lemm}
\begin{proof}
If $l=1$, then 
\begin{align*}
     \|L\|_{L^2}^2  \leq C \sum_{j=0}^{k-1} D_j\, .
\end{align*}
Assume $l\geq 2$. Without loss of generality, assume $|\beta_1|\leq ...\leq |\beta_l|$. For any $1\leq j\leq l-1$, (\ref{gn}) and (\ref{L infinity bound}) give
\begin{align}\label{lemma h^m estimate inequality}
    \| \partial_x^{\beta_j}\psi\|_{L^{p}} \leq C \|\psi\|_{H^{m}} \leq C \epe <1 \, ,
\end{align}
for any $2\leq p<\frac{2d}{d-2m + 2|\beta_j|}:=p_j$. 
We will consider 3 cases.

\noindent
\textit{1) case $k+1-|\beta_l|\leq m-1$}. 
By (\ref{gn}), we have
\begin{align*}
    \| \partial_x^{\beta_l}\psi\|_{L^{p}} \leq C \|\psi\|_{H^{k+1}} \, ,
\end{align*}
for any $2\leq p<\frac{2d}{d-2k-2 + 2|\beta_l|}:=p_l$.

Since
\begin{align*}
    \sum_{j=1}^l \frac{1}{p_j} \leq \frac{1}{2} + \frac{(d-2m)(l-1)}{2d}<\frac{1}{2} \; \text{ and } \;
    \frac{l}{2} \geq \frac{1}{2} \, ,
\end{align*}
there exist $p_1^*,...,p_l^*$ such that
\begin{align*}
    2\leq p_j^* < p_j \; \text{ and } \; \sum_{j=1}^l \frac{1}{p_j^*} = \frac{1}{2} \, .
\end{align*}
Then H\"{o}lder's inequality and (\ref{lemma h^m estimate inequality}) give
\begin{align*}
    \|L\|_{L^2}^2 \leq C \Pi_{j=1}^l \|\partial_x^{\beta_j}\psi\|_{L^{p_j^*}}^2 \leq (C \epe)^{2(l-1)} \|\psi\|_{H^{k+1}}^2 \leq C \epe \sum_{j=0}^k D_j \, .
\end{align*}

\noindent
\textit{2) case $k+1-|\beta_l|=m$ and $d\geq 2$}. Then $k=m$ and $|\beta_j|=1$ for any $1\leq j\leq l$. By (\ref{gn extreme}), we have
\begin{align*}
    \| \partial_x^{\beta_l}\psi\|_{L^{p}} \leq C \|\psi\|_{H^{m+1}} \, ,
\end{align*}
for any $2\leq p< \infty:=p_l$.
Since $1/p_l = 0$ and $l-1\leq m$,
\begin{align*}
    \sum_{j=1}^l \frac{1}{p_j} \leq \frac{(d-2m+2)m}{2d} \leq \frac{d+1}{4d}<\frac{1}{2} \; \text{ and } \; \frac{l}{2} \geq \frac{1}{2} \, .
\end{align*}
Hence, there exist $p_1^*,...,p_l^*$ such that
\begin{align*}
    2\leq p_j^* < p_j \; \text{ and } \; \sum_{j=1}^l \frac{1}{p_j^*} = \frac{1}{2} \, .
\end{align*}
Then H\"{o}lder's inequality and (\ref{lemma h^m estimate inequality}) give
\begin{align*}
    \|L\|_{L^2}^2\leq C \Pi_{j=1}^l \|\partial_x^{\beta_j}\psi\|_{L^{p_j^*}}^2 \leq (C \epe)^{2(l-1)} \|\psi\|_{H^{m+1}}^2 \leq C \epe \sum_{j=0}^{m} D_j \, .
\end{align*}

\noindent
\textit{3) case $k+1-|\beta_l|=m$ and $d = 1$}. Then $k=m=1$, $l=2$, and $|\beta_1|=|\beta_2|=1$. H\"{o}lder's inequality, Gagliardo-Nirenberg inequality and (\ref{L infinity bound}) give
\begin{align*}
    \|L\|_{L^2}^2 & \leq C \|\partial_x \psi\|_{L^2}^2 \|\partial_x \psi\|_{L^\infty}^2   \leq C \epe \|\psi\|_{H^2}^2\leq C \epe (D_1+D_0)\, .
\end{align*}

\end{proof}

\subsection{$H^m$ estimates}
We prove the $H^m$ estimates by induction.
\begin{lemm} \label{lemma H^m estimate}
    For any $\ul \in \mathbb{R}^n$, there exist $\epw, \epe, C, \Lambda_{S_1}, \Lambda_{W_n}, \Tilde{C}_1, \Tilde{C}_n>0$ such that the following is true. 

    \noindent
    Assume Assumption \ref{assumption}. Then for any $t \in [0,T]$,
    \begin{equation}\label{h^m contraction ineq 1}
        \begin{aligned}
            &\|U(t,\cdot)-\utl(t,\cdot)\|_{H^m(\mathbb{R}\times {\mathbb{T}^{d-1}})} ^ 2
            + \int_0^t \Bigl(\sum_{k=0}^m D_k(U) + G_{S_1}(U) + G_{W_n}(U) + {Y}\Bigl) \,ds 
            \\ &\leq  C \|U_0-\utl_0\|_{H^m(\mathbb{R}\times {\mathbb{T}^{d-1}})} ^2 + C E^2 \, ,
        \end{aligned}
    \end{equation}
    and for any $t \in [0,T]$ and any $\beta \in \mathbb{N}^d$ such that $1\leq |\beta|\leq m$,
    \begin{equation}\label{h^m contraction ineq 2}
        \begin{aligned}
            &\int_0^t \Big|\frac{d}{dt} \int \big|\partial_x^\beta (U-\utl)\big|^2 \, dx\Big| \, ds
            \\&\leq C \int_0^t \Bigl(\sum_{k=0}^m D_k(U) + G_{S_1}(U) + G_{W_n}(U) + Y \Bigl) \, ds + C E^2\, ,
        \end{aligned}
    \end{equation}
    where $m$ is defined in (\ref{value of m}) and $D_k, G_{S_1}, G_{W_n}, Y, E$ are defined in \textup{(\ref{def of D, D1, G_S, G_R})}.
\end{lemm}
\begin{proof}
We will prove the lemma by induction. Recall the $L^2$ estimates shown in Lemma \ref{lemma L^2 estimate}. Let $1\leq k\leq m$.
We assume the inequality is true for $k-1$, i.e., 
\begin{align*}
    \begin{split}
        &\|U(t,\cdot)-\utl(t,\cdot)\|_{H^{k-1}(\mathbb{R}\times {\mathbb{T}^{d-1}})} ^ 2
        + \int_0^t \Bigl(\sum_{j=0}^{k-1} D_j(U) + G_{S_1}(U) + G_{W_n}(U) + Y \Bigl) \, ds 
        \\ &\leq  C \|U_0-\utl_0\|_{H^{k-1}(\mathbb{R}\times {\mathbb{T}^{d-1}})} ^2 + C E^2 \, .
    \end{split}
\end{align*}
We want to show the inequality is true for $k$, i.e.,
\begin{align}\label{h^m contraction ineq 3}
    \begin{split}
        &\|U(t,\cdot)-\utl(t,\cdot)\|_{H^k(\mathbb{R}\times {\mathbb{T}^{d-1}})} ^ 2
        + \int_0^t \Bigl (\sum_{j=0}^k D_j(U) + G_{S_1}(U) + G_{W_n}(U) + Y \Bigl) \, ds 
        \\ &\leq  C \|U_0-\utl_0\|_{H^k(\mathbb{R}\times {\mathbb{T}^{d-1}})} ^2 + C E^2 \, .
    \end{split}
\end{align}
Note the constants $C>0$ in this proof depend on $B_j,f,\eta,\ul,k$, but the dependency on $k$ does not matter as $m$ is fixed and finite.

Recall the definition of $\psi$ in (\ref{notataion of perturbation}). Let $\alpha_k \in \mathbb{N}^d$ be such that $|\alpha_k|=k$.
We take $\partial_x^{\alpha_k}$ on both sides of (\ref{eq u and u tilde}) and get
\begin{align*}
    &\partial_t \partial_x^{\alpha_k} \psi + \partial_{x}^{{\alpha_k}} \partial_{x_1} \bigl(f(U)-f(\utl)\bigl) + \sum_{j=2}^{d}\partial_{x}^{{\alpha_k}} \partial_{x_j} \bigl(g_j(U)-g_j(\utl)\bigl) \\
    &= \sum_{j=1}^d \partial_{x}^{\alpha_k} \partial_{x_j} \bigl(B_j(U)\partial_{x_j} \eta'  (U) - B_j(\utl)\partial_{x_j} \eta'  (\utl)\bigl)  - \partial_{x}^{\alpha_k}Z - \partial_{x}^{\alpha_k} E_1 - \partial_x^{\alpha_k} E_2 \, .
\end{align*}
Then we multiply both sides by $\partial_x^{\alpha_k}\psi$ and take integration w.r.t. $x$. Recall (\ref{notation of alpha}). By integration by parts, we have
\begin{equation*}
    \begin{aligned}
        &\frac{d}{dt} \int \frac{|\partial_x^{\alpha_k}\psi|^2 }{2} \, dx +\sum_{j=1}^d \int \partial_{x}^{\alpha_k} \bigl(B_j(U)\eta''(U)  \partial_{x_j} U - B_j(\utl) \eta''  (\utl)\partial_{x_j} \utl\bigl) \,  \partial_x^{\alpha_k+e_j}\psi \, dx
        \\ &= \int \partial_x^{\alpha_k} \bigl(f(U)-f(\utl)\bigl) \,\partial_x^{\alpha_k+e_1}\psi \,dx + \sum_{j=2}^d \int \partial_x^{\alpha_k} \bigl(g_j(U)-g_j(\utl)\bigl) \,\partial_x^{\alpha_k+e_j}\psi \, dx
        \\ & \quad -  \int \partial_x^{\alpha_k} Z \,\partial_x^{\alpha_k}\psi \, dx
        -\int \partial_{x}^{\alpha_k} E_1 \,  \partial_{x}^{\alpha_k}\psi \,dx -\int \partial_{x}^{\alpha_k} E_2 \, \partial_{x}^{\alpha_k}\psi \,dx\, ,
    \end{aligned}
\end{equation*}
where $e_j \in \mathbb{N}^d$ is the multi-index whose j-th component is $1$ and all other components are 0.

Let $0< \tau< 1$ to be chosen later in (\ref{choice of gamma}). Recall (\ref{notation of alpha}).
By Young's inequality and integration by parts, we have
\begin{align*}
    \frac{d}{dt} \int \frac{|\partial_x^{\alpha_k}\psi|^2 }{2} \, dx + I_1
    \leq \tau D_k + C D_{k-1} + \frac{C}{\tau}(I_2+I_3+I_4+I_5) \, ,
\end{align*}
where
\begin{equation*}
    \begin{aligned}
        I_1&= \sum_{j=1}^d \int \partial_{x}^{\alpha_k} \bigl(B_j(U)\eta''(U)  \partial_{x_j} U - B_j(\utl) \eta''  (\utl)\partial_{x_j} \utl\bigl) \, \partial_x^{\alpha_k+e_j}\psi \, dx\, , \\
        I_2&= \int \big|\partial_x^{\alpha_k} \bigl(f(U)-f(\utl)\bigl)\big|^2 \, dx + \sum_{j=2}^d \int \big|\partial_x^{\alpha_k} \bigl(g_j(U)-g_j(\utl)\bigl)\big|^2\, dx \, , \\
        I_3&= Y \, ,\\
        I_4&=\int |\partial_{x}^{\alpha_k} E_1|^2 \, dx \, ,\\
        I_5&= 
        \begin{cases}
            \int |\partial_{x}^{\alpha_k-e_1} E_2|^2 \, dx\, , \; \textup{if } (\alpha_k)_1 \textup{ is odd and } W_n = R_n\,, \\
            \int |\partial_{x}^{\alpha_k} E_2|^2  \, dx\, , \; \textup{otherwise.} 
        \end{cases}
    \end{aligned}
\end{equation*}
Lemma \ref{lemm reorganization} and Lemma \ref{lemm estimate in H^m contraction} give
\begin{equation*}
    \begin{aligned}
        |I_1 - \Tilde{I}_1| 
        &\leq (\tau + \frac{C}{\tau}\epe) D_k + \frac{C}{\tau} (G_{S_1} + G_{W_n} +\sum_{j=0}^{k-1} D_j) \, , \\
        I_2 &\leq C \epe D_k + C (G_{S_1}+G_{W_n} + \sum_{j=0}^{k-1} D_j) \, , 
    \end{aligned}
\end{equation*}
where
\begin{align*}
    \Tilde{I}_1 = \sum_{j=1}^d \int \bigl(B_j(U)\eta''(U)\partial_x^{\alpha_k+e_j}U -B_j(\utl)\eta''(\utl)\partial_x^{\alpha_k+e_j}\utl\bigl)\,  \partial_x^{\alpha_k+e_j}\psi \,dx\, .
\end{align*}
We first evaluate $\Tilde{I}_1$. We have
\begin{equation*}
    \Tilde{I}_1 = \Tilde{D}_k + \Tilde{I}_{11} +  \Tilde{I}_{12} \, ,
\end{equation*}
where
\begin{align*}
    \Tilde{D}_k &= \sum_{j=1}^d\int  B_j(\ul)\eta''(\ul) \partial_x^{\alpha_k+e_j}\psi \, \partial_x^{\alpha_k+e_j}\psi \,dx\, , \\
    \Tilde{I}_{11} &= \sum_{j=1}^d\int \bigl(B_j(U)\eta''(U)- B_j(\ul)\eta''(\ul)\bigl) \partial_x^{\alpha_k+e_j}\psi \,\partial_x^{\alpha_k+e_j}\psi \,dx\, ,\\
    \Tilde{I}_{12}  &= \sum_{j=1}^d\int \bigl(B_j(U)\eta''(U) - B_j(\utl)\eta''(\utl)\bigl)\partial_x^{\alpha_k+e_j}\utl \,\partial_x^{\alpha_k+e_j} \psi\,dx \, .
\end{align*}
As $B_j$ are positive definite and $\eta''$ is strictly convex,
\begin{align*}
    \Tilde{D}_k \geq C \sum_{j=1}^d \|\partial_x^{\alpha_k+e_j}\psi\|_{L^2}^2 \, .
\end{align*}
By (\ref{L infinity}) and Young's inequality, we get
\begin{align*}
    |\Tilde{I}_{11}| + |\Tilde{I}_{12}| \leq C (\epe+\epw) D_k + {C}(G_{S_1} + G_{W_n}) \, .
\end{align*}
Therefore, we have
\begin{align*}
    |I_1 -\Tilde{D}_k| \leq \bigl(\tau+\frac{C}{\tau}(\epe+\epw)\bigl) D_k + \frac{C}{\tau}(G_{S_1} +G_{W_n}+ \sum_{j=0}^{k-1} D_j) \, .
\end{align*}
If $W_n=S_n$, then Lemma \ref{lemm reorganization} gives
\begin{align*}
    I_4 + I_5 &\leq C \Bigl(\big\||\st-\um| |\partial_{x_1} \so|\big\|_{L^2}^2 + \big\||\so-\um| |\partial_{x_1} \st|\big\|_{L^2}^2  \\ &\quad + \big\||\partial_{x_1} \st||\partial_{x_1} \so |\big\|_{L^2}^2\Bigl)
    \\ &=: \mathcal{E}_{S_n}  \, .   
\end{align*}
If $W_n = R_n$, then Lemma \ref{lemm reorganization} gives
\begin{align*}
    I_4 + I_5 &\leq C \Bigl(\|\partial_{x_1}^{2j_*} R_n\|_{L^2}^2  + \|\partial_{x_1} R_n\|_{L^4}^4 + \big\||\partial_{x_1} R_n || \partial_{x_1} \st|\big\|_{L^2}^2 \\
    &\quad + \big\||R_n-\um||\partial_{x_1} \st|\big\|_{L^2}^2 + \big\||\st-\um||\partial_{x_1} R_n|\big\|_{L^2}^2\Bigl) 
    \\ &=: \mathcal{E}_{R_n} \, ,
\end{align*}
for some $j_* \in \mathbb{N}^*$.

We get
\begin{equation}\label{h^m contraction ineq 4}
    \begin{aligned}
        \frac{d}{dt} \int \frac{|\partial_{x_1}^{\alpha_k}\psi|^2 }{2} \,dx+ C \sum_{j=1}^d \|\partial_x^{\alpha_k+e_j}\psi\|_{L^2}^2
        \leq \bigl(2\tau+\frac{C}{\tau}(\epe+\epw)\bigl)  D_k \\ + \frac{C}{\tau}(G_{S_1}+G_{W_n}+ \sum_{j=0}^{k-1} D_j + Y + \mathcal{E}_{W_n})\, .
    \end{aligned}
\end{equation}
Let
\begin{align*}
    \mathcal{A}_k = \bigl\{\alpha \in \mathbb{N}^d : |\alpha|=k\bigl\} \, .
\end{align*}
As $\alpha_k$ is arbitrary, we get
\begin{equation*}
    \begin{aligned}
        \frac{d}{dt} \sum_{\alpha\in \mathcal{A}_k}\int \frac{|\partial_{x_1}^{\alpha}\psi|^2 }{2} \,dx+ C_1 D_k
        \leq \bigl(2|\mathcal{A}_k|\tau+\frac{C_2}{\tau}(\epe+\epw)\bigl)  D_k \\ + \frac{C}{\tau}(G_{S_1}+G_{W_n}+ \sum_{j=0}^{k-1} D_j + Y + \mathcal{E}_{W_n})
        \, ,
    \end{aligned}
\end{equation*}
where $C_1, C_2>0$ are constants that depend only on $B_j,f,\eta,\ul,k$.
We take
\begin{equation}\label{choice of gamma}
    \tau = \textup{min}\bigl\{\frac{C_1}{6|\mathcal{A}_k|},1\bigl\} \, \textup{ and } \, \epe, \epw \leq \frac{C_1 \tau}{6 C_2} \, .
\end{equation}
Finally, we take integration over time. By (\ref{rarefaction time bound}) and Lemma \ref{interaction bound}, we get the estimates (\ref{h^m contraction ineq 3}). Hence, we prove (\ref{h^m contraction ineq 1}). The inequality (\ref{h^m contraction ineq 4}) also gives
\begin{align*}
    \big|\frac{d}{dt} \int {|\partial_{x_1}^{\alpha_k}\psi|^2 } \,dx \big|
    \leq {C}(G_{S_1}+G_{W_n}+ \sum_{j=0}^{k} D_j + Y + \mathcal{E}_{W_n}) \, .
\end{align*}
By taking integration over time, (\ref{rarefaction time bound}) and Lemma \ref{interaction bound} give (\ref{h^m contraction ineq 2}).
\end{proof}

\section{Appendix}
\subsection{Proof of Proposition \ref{prop BNS}}
The entropy of the 3-D barotropic Brenner-Navier-Stokes equations is 
\begin{align*}
    \eta= \rho \frac{|u|^2}{2} + Q  \, .
\end{align*}
We let 
\begin{align*}
    U=
    \begin{pmatrix}
    \rho \\ \rho u 
    \end{pmatrix}
    \, .
\end{align*}
We rewrite the 3-D barotropic Brenner-Navier-Stokes equations (\ref{BNS eqs}) w.r.t. $U$:
\begin{equation*}
    \begin{cases}
        \partial_t \rho + div(\rho u) = \Delta Q' \, , \\
        \partial_t \rho u + div(\rho u \otimes u) + \nabla \rho^\gamma = \nu\Delta u + div(u\otimes \nabla Q ') \, .
    \end{cases}
\end{equation*}
As
\begin{align*}
    \eta'(U)=
    \begin{pmatrix}
    Q'-\frac{|u|^2}{2} \\ u 
    \end{pmatrix}
    \, ,
\end{align*}
we get
\begin{align*}
    \Delta Q' = \Delta \eta_1 ' + \Delta \frac{|u|^2}{2} \, , 
\end{align*}
and for any $1\leq j\leq 3$,
\begin{align*}
    div(u_j \otimes \nabla Q ') &= \sum_{k=1}^3 \partial_k(u_j \partial_k Q') \\
    &= \sum_{k=1}^3 \partial_k(u_j \partial_k \eta_1') + \partial_k(u_j \partial_k \frac{|u|^2}{2}) \, .
\end{align*}
Hence, for any $1\leq j\leq 3$,
\begin{align*}
    B_j(U) = \begin{pmatrix}
    1& u_1 &u_2 &u_3 \\ u_1 & u_1^2 + \nu & u_1 u_2 & u_1 u_3 \\
    u_2 & u_1 u_2 & u_2^2 + \nu & u_2 u_3 \\ u_3 & u_1 u_3 & u_2 u_3 & u_3^2 + \nu
    \end{pmatrix}
    \, .
\end{align*}
Since the determinant of $B_j$ is $\nu^3$, $B_j$ is positive definite.

\bibliographystyle{abbrv}
\bibliography{refs}

\end{document}